\documentclass[12pt, reqno, a4paper,oneside]{amsart}

\rmfamily{\fontsize{12pt}{\baselineskip}\selectfont}

\usepackage{graphicx}
\usepackage{amsfonts, amsmath, amssymb, amsbsy, amsthm}
\usepackage{bm}
\usepackage{color}
\usepackage{xcolor}
\usepackage{hyperref}
\usepackage{mathrsfs}
\usepackage{longtable}
\usepackage[top=2.5cm,bottom=2.0cm,left=2.0cm,right=2.0cm]{geometry}
\usepackage{epstopdf}
\usepackage{stmaryrd}
\usepackage[normalem]{ulem}
\usepackage{cancel}
\usepackage{scalerel}
\usepackage{enumitem}
\usepackage{algorithm}
\usepackage{subfig}
\usepackage{algorithm, algorithmicx}
\usepackage[noend]{algpseudocode}
\usepackage{lineno}
\usepackage{enumitem}
\usepackage{diagbox}
\usepackage{multirow}
\renewcommand{\algorithmiccomment}[1]{\bgroup\hfill//~#1\egroup}

\usepackage{environments}
\numberwithin{theorem}{section}
\numberwithin{equation}{section}

\newcommand\tbbint{{-\mkern -16mu\int}}

\newcommand\dbbint{{-\mkern -19mu\int}}

\newcommand\bbint{
{\mathchoice{\dbbint}{\tbbint}{\tbbint}{\tbbint}}
}

\renewcommand\arraystretch{1.5}

\def\per{{\rm per}}
\def\mA{{\sf A}}
\def\T{\mathcal{T}}
\def\D{\mathcal{D}}
\def\Z{\mathbb{Z}}

\def\assERL{\textnormal{\bf (RL)}}

\def\asS{\textnormal{\bf (S)}}

\def\Vhom{V^{\rm h}}

\def\E{\mathcal{E}}

\def\p0{\pmb{0}}

\def\Us{\mathscr{U}}

\def\UsH{{\mathscr{U}}^{1,2}}

\def\efit{\varepsilon^{\rm E}}
\def\ffit{\varepsilon^{\rm F}}
\def\fcfit{\varepsilon^{\rm FC}_{\rm hom}}

\def\L{\Lambda}
\def\R{\mathbb{R}}
\def\Rl{\mathcal{R}}

\def\Adm{{\rm Adm}}
\def\<{\langle}
\renewcommand{\>}{\rangle}

\def\Lhom{\L^{\rm h}}

\def\F{\mathcal{F}}
\def\N{\mathbb{N}}

\def\Adm{{\rm Adm}}
\def\Admu{\mathscr{A}}

\def\fc{f_{\rm c}}
\def\Rc{R_{\rm c}}
\def\dx{\,{\rm d}x}

\def\ds{\,{\rm d}s}

\def\n{\mathfrak{n}}

\def\Lhom{\Lambda^{\rm hom}}

\definecolor{yscol}{HTML}{6622AA}

\begin{document}
\title[Generalisation analysis of MLIPs]{A framework for a generalisation analysis of machine-learned interatomic potentials}

\author{Christoph Ortner}
\address{Christoph Ortner\\
Department of Mathematics\\
University of British Columbia\\
1984 Mathematics Road\\
Vancouver, British Columbia\\
Canada
}
\email{ortner@math.ubc.ca}

\author{Yangshuai Wang}
\address{Yangshuai Wang\\
Department of Mathematics\\
University of British Columbia\\
1984 Mathematics Road\\
Vancouver, British Columbia\\
Canada
}
\email{yswang2021@math.ubc.ca}

\date{\today}




\begin{abstract}
	Machine-learned interatomic potentials (MLIPs) and force fields (i.e. interaction laws for atoms and molecules) are typically trained on limited data-sets that cover only a very small section of the full space of possible input structures. MLIPs are nevertheless capable of making accurate predictions of forces and energies in simulations involving (seemingly) much more complex structures. In this article we propose a framework within which this kind of generalisation can be rigorously understood. As a prototypical example, we apply the framework to the case of simulating point defects in a crystalline solid. Here, we demonstrate how the accuracy of the simulation depends explicitly on the size of the training structures, on the kind of observations (e.g., energies, forces, force constants, virials) to which the model has been fitted, and on the fit accuracy. The new theoretical insights we gain partially justify current best practices in the MLIP literature and in addition suggest a new approach to the collection of training data and the design of loss functions.
\end{abstract}

\maketitle


\section{Introduction}
\label{sec:intro}
A key question in machine learning tasks is to understand how well a trained model generalizes to inputs outside of the training data. This is particularly challenging in scientific machine learning where one oftentimes requires generalisation to inputs {\em very far from} training data. The present work is concerned with sketching out a multiscale numerical analysis framework suitable to study this scenario.

We will focus in particular on atomistic mechanics simulations using machine learned interatomic potentials (MLIPs)~\cite{2019-ship1, Bart10, behler07, Braams09, Drautz19, Shapeev16}. The success of molecular simulation relies on the accuracy and efficiency of the interatomic force models. The two main approaches to computing interatomic forces are ab initio electronic structure models \cite{galli1992large, kohanoff2006electronic, kotliar2006electronic, saad2010numerical} and purely mechanistic models  \cite{Daw1984a, finnis03, LJ, Stillinger1985Computer}. The former are computationally prohibitive while the latter often provide insufficient accuracy. The achievement of MLIPs is to provide a classes of models with tunable accuracy/efficiency ratio which promises to bridge the significant gap in accuracy and capability between ab initio electronic structure models and classical mechanistic models (empirical potentials).

In this work we will study by analytical (as opposed to statistical) methods how the choice of training data and the accuracy of the fit to that training data affect the accuracy of predictions. The key challenge we hope to better understand is the following: training data for MLIPs is obtained from ab initio electronic structure simulations. Due to the high computational cost of these models, only small computational domains (structures) containing at most hundreds of atoms are used. However, predictions during simulations are performed on much larger domains often containing hundreds of thousand or even millions of atoms. 

As a prototype application we consider classes of structures containing crystalline defects, and to keep the notational and technical burden to a minimum we further restrict the present work to point defects only. In this setting, training domains would typically be small cells containing a single defect, while simulations would be performed on much larger domains containing potentially many copies of the defects trained on. In this situation, the basic intution is clear: the trained MLIP has already seen local snapshots of the structure on which it is predicting energies or interatomic forces, and due to its functional form that is indeed restricted to only local interactions it is therefore able to to make accurate predictions. 

Yet, the details are subtle and warrant a deeper look: How does the prediction error depend on the size of the training domain? How should the various observations we make (energies, forces, virials) be weighted? Which of these observations provide the dominant contribution to the prediction error? Thus, we see that even this highly simplified setting leads to interesting questions that can significantly inform the design of parameter estimation schemes.

\subsection{Outline}
We focus on multiple point defects embedded in a periodic homogeneous host crystal, where a rigorous numerical analysis approach is feasible. The atomistic equilibration problem for a single crystalline defect in this context is a well-defined variational problem~\cite{chen19, Ehrlacher16}. We review the framework and adapt it to the case of multiple point defects considered in this work along the line of~\cite{2014-dislift} in Section~\ref{sec:equilibration}. This requires in particular a new existence and stability result (Theorem~\ref{them:existence}) for general configurations of multiple point defects. Our generalisation analysis heavily relies on this result since it characterises the structure of the equilibrium structures.

To propose a framework within which the generalisation can be rigorously understood, we investigate the error propagation from fitting MLIPs on a small training domain to predicting the material properties (e.g., defect geometry and formation energy) on a large simulation domain in Section \ref{sec:sub:main}. We demonstrate how the accuracy of the material properties in the simulation depends explicitly on the size of the training structures, on the kind of observations (e.g., energies, forces, force constants) to which the model has been fitted, and on the fit accuracy. 
Explicit theoretical convergence rates are summarized in Theorem~\ref{them:geometry} and Table~\ref{table-e-mix}.  

We then propose a concrete implementation of MLIPs inspired by our generalisation analysis to confirm the analytical error estimates on several model problems in Section~\ref{sec:numer}. 

Finally, we will discuss further consequences and limitations our work  in Section \ref{sec:conclu}. For example, a generalisation to other ``simple defects'' such as straight dislocation lines appears straightforward. In more complex scenarios, such as curved dislocations or extrapolating on grain boundary structures with distinct coordination environment there are additional challenges that our analysis does not cover even heuristically and requires significant additional ideas.



\subsection{Notation}
We use the symbol $\langle\cdot,\cdot\rangle$ to denote an abstract duality
pairing between a Banach space and its dual space. The symbol $|\cdot|$ normally
denotes the Euclidean or Frobenius norm, while $\|\cdot\|$ denotes an operator
norm.
For a finite set $A$, we will use $\#A$ to denote the cardinality of $A$.
For the sake of brevity of notation, we will denote $A\backslash\{a\}$ by
$A\backslash a$, and $\{b-a~\vert ~b\in A\}$ by $A-a$.
For $E \in C^2(X)$, the first and second variations are denoted by
$\<\delta E(u), v\>$ and $\<\delta^2 E(u) v, w\>$ for $u,v,w\in X$.
For $j\in\N$, ${\bm{g}}\in (\R^d)^A$, and $V \in C^j\big((\R^d)^A\big)$, we define the notation
\begin{eqnarray*}
	V_{,{\bm \rho}}\big({\bm g}\big) :=
	\frac{\partial^j V\big({\bm g}\big)}
	{\partial {\bm g}_{\rho_1}\cdots\partial{\bm g}_{\rho_j}}
	\qquad{\rm for}\quad{\bm \rho}=(\rho_1, \ldots, \rho_j)\in A^{j}.
\end{eqnarray*}
The symbol $C$ denotes generic positive constant that may change from one line
of an estimate to the next. When estimating rates of decay or convergence, $C$
will always remain independent of the system size, the configuration of the lattice and the the test functions. The dependence of $C$ will be normally clear from the context or stated explicitly.
The closed ball with radius $r>0$ and center $x$ is denoted by $B_r(x)$, or $B_r$ if the center is the origin.

\section{Background: Equilibration of crystalline defects}
\label{sec:equilibration}
\setcounter{equation}{0}

A rigorous framework for modelling the geometric equilibrium of crystalline defects has been developed in~\cite{chen18, chen19, Ehrlacher16, co2020}.
These works formulate the equilibration of a single crystalline defect as a variational problem in a discrete energy space and establish qualitatively sharp far-field decay estimates for the corresponding equilibrium configuration. 
We will review the framework and adapt it to the case of multiple point defects considered in this work along the lines of~\cite{2014-dislift}. This will provide the analytic foundation of our generalisation analysis.
For the sake of simplicity of presentation, we will skip over some technical details but fill these gaps in Section~\ref{sec:sub:pre}.

Let $d\in\{2,3\}$ be the (effective) dimension of the system. A homogeneous crystal reference configuration is given by the Bravais lattice
$\Lhom=\mathsf{A}\Z^d$, for some non-singular matrix $\mathsf{A} \in \mathbb{R}^{d \times d}$.
We admit only single-species Bravais lattices. There are no conceptual obstacles to generalising our work to multi-lattices, however, the technical details become more involved.
The reference configuration with defects is a set $\L \subset \R^d$. The mismatch between $\L$ and $\Lhom$ represents possible defected configurations. 
In this paper, we consider multiple point defects in a finite domain with periodic boundary conditions.
To that end, let $\mathsf{B}=(b_1, \ldots, b_d) \in \R^{d \times d}$ invertible such that $b_i \in \mathsf{A}\Z^d$. We denote the continuous cell by $\Omega_N:=\mathsf{B}(-N/2, N/2]^d$. For a sufficiently large $N \in \N$, let 
\[
\L_N:=\L\cap\Omega_N \quad \textrm{and} \quad \L_N^{\per}:= \bigcup_{\alpha \in N\Z^d}(\mathsf{B}\alpha + \L_N),
\]
where $\L_N$ is the periodic computational domain and $\L_N^{\per}$ is the periodically repeated domain. 

We consider $n_{\D}$ point defects in $\L_N$, e.g. vacancies or interstitials, located at $\ell_i\in\Omega_N$ for $i=1,\ldots,n_{\D}$. Let $\D:=\{\ell_i\}_{i=1}^{n_{\D}}$ be a set of the positions of these defect cores in $\L_N$. We assume that the defect cores are localized, that is, there exists $R_{\rm def}>0$ such that $\L_N \setminus \cup^{n_{\D}}_{i=1} B_{R_{\rm def}}(\ell_i) = (\Lhom \cap \Omega_N) \setminus \cup^{n_{\D}}_{i=1} B_{R_{\rm def}}(\ell_i)$.
We define the minimum separation distance of $\D$ by 
$$
L_{\D} := \inf \Big\{ |\ell'_i - \ell'_j|~\big|~\ell'_i, \ell'_j \in \bigcup_{k=1}^{n_{\D}}\bigcup_{\alpha\in N\Z^d}(\mathsf{B}\alpha+\ell_k), i\neq j \Big\}.
$$ 
We assume $L_{\D} \ll N$ throughout this paper and we refer to Figure \ref{fig:illustration} for a two dimensional example with $\mA$ defined by \eqref{eq:mA} specifying a triangular lattice.
\begin{figure}[!htb]
    \centering
    \includegraphics[height=5cm]{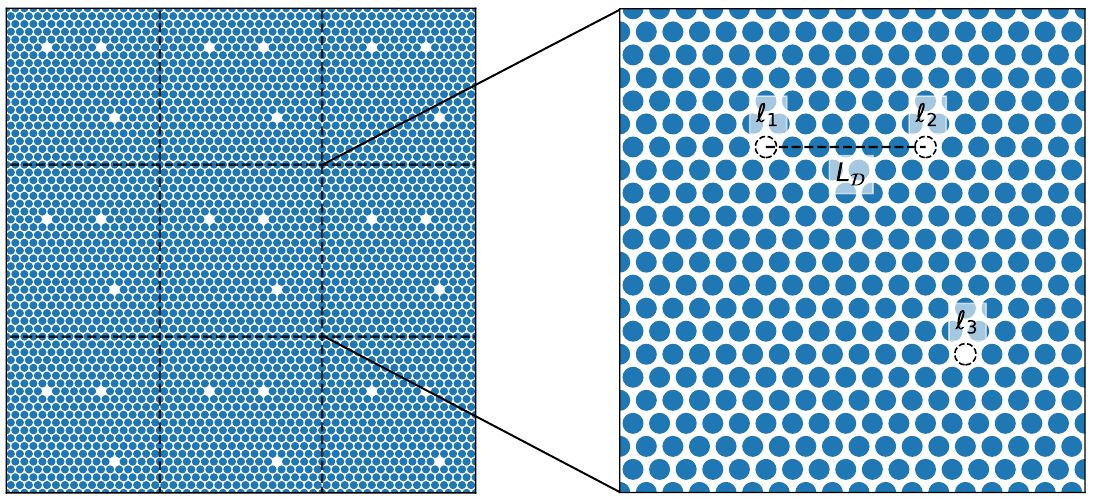}
    \caption{Illustration of $\L_N^{\rm per}$ (left), $\L_N$ and $\mathcal{D}$ (right) for three vacancies in a 2D triangular lattice.}
    \label{fig:illustration}
\end{figure}

A {\it displacement} field of the periodically repeated lattice $\L_N^{\rm per}$ is a map $u:\L_N^{\rm per}\rightarrow\R^d$.
For $\ell\in\L_N^{\rm per}$ and $\rho\in\L_N^{\rm per}-\ell$, we define the finite difference
$D_\rho u(\ell) := u(\ell+\rho) - u(\ell)$. For a subset $\Rl \subset \L_N^{\rm per}-\ell$, we
define $D_{\Rl} u(\ell) := (D_\rho u(\ell))_{\rho\in\Rl}$, and we consider $Du(\ell) := D_{\L_N^{\rm per}-\ell} u(\ell)$ to be a finite-difference stencil with infinite range.
For a stencil $Du(\ell)$, we define the stencil norms
\begin{align}\label{eq: nn norm}
\big|Du(\ell)\big|_{\mathcal{N}} := \bigg( \sum_{\rho\in \mathcal{N}(\ell) - \ell} \big|D_\rho u(\ell)\big|^2 \bigg)^{1/2}  
\quad{\rm and}\quad
\|Du\|_{\ell^2_{\mathcal{N}}(\L_N)} := \bigg( \sum_{\ell \in \L_N} |Du(\ell)|_{\mathcal{N}}^2 \bigg)^{1/2},
\end{align}
where $\mathcal{N}(\ell)$ defined by \eqref{def1:Nl} is the set containing nearest neighbours of site $\ell$.

The site potential is a collection of mappings $V_{\ell}:(\R^d)^{\L-\ell}\rightarrow\R$, which represent the energy distributed to each atomic site. 
To simplify the notation we assume that $V_{\ell}({\bf 0})=0$ for all $\ell$, which is equivalent to considering a potential energy-difference.
We state the assumptions on the regularity and locality of the site potentials in Section \ref{sec:sub:pre} and refer to~\cite[\S 2.3 and \S 4]{chen19} for a detailed discussion of those assumptions.
%
If $\L = \Lhom$, we denote the site potential by $\Vhom:(\R^d)^{\Lhom \setminus 0}\rightarrow\R$.

%
We define the space of periodic displacements to be
\[
\Us^{\per}_N := \{u:\L_N^{\per} \rightarrow \R^d~|~u(\ell+\mathsf{B}\alpha) = u(\ell) ~\textrm{for}~\alpha \in N\Z^d\}.
\]
%
For $u \in \Us^{\per}_{N}$ and $N$ sufficiently large, the periodic energy functional reads 
\begin{eqnarray}\label{energy-difference-per}
\E(u) := \sum_{\ell\in\Lambda_N} V_{\ell}\big(Du(\ell)\big).
\end{eqnarray}
An equilibrium defect geometry is obtained by solving
\begin{align}\label{eq:variational-problem-per}
&\bar{u} \in \arg\min \big\{ \E(u), u \in \Admu \big\}, \\
\textrm{where} \quad &\Admu := \{ u \in \Us_N^{\per}: x + u \in \Adm_0(\L^{\rm per}_N)  \} \nonumber 
\end{align}
is the admissible set with $x:\L_N^{\rm per}\rightarrow\R, x(\ell):=\ell$ and $\Adm_0(\L^{\rm per}_N)$ defined in \eqref{eq:admiss} is represents a constraint preventing the collision of atoms. 

For the purpose of error analysis we will need a strong stability condition~\cite{chen19, Ehrlacher16} as well as qualitative information about the equilibrium. To that end, we assume the following:
\begin{flushleft}\label{as:LSs}
	\asS \quad There exists a single strongly stable core $\bar{u}^{\rm core}$ defined by \eqref{eq:variational-problem} in the infinite lattice $\L$.
\end{flushleft}
For the sake of simplicity of presentation, we postpone the rigorous formulation of \asS to Section~\ref{sec:sub:pre}, Equation~\eqref{eq:def_S}.

The following result shows that, given $n_{\D}$ point defects in $\L_N$, there exists a strongly stable equilibrium of \eqref{eq:variational-problem-per} as long as \asS~holds and the core positions satisfy a minimum separation criterion from each other. The detailed proof is given in Section \ref{sec:sub:thm1}.

\begin{theorem}\label{them:existence}
Suppose that \asS~holds. Then, for each $n_{\D} \in \N$, there exists a constant $L_0>0$ such that for any core configuration $\D$ satisfying $L_{\D} \geq L_0$, there exists a strongly stable equilibrium of \eqref{eq:variational-problem-per} and it can be written as
\begin{eqnarray}
\bar{u}(\ell) = \sum_{\ell_i \in \D} \Pi_{R} \bar{u}^{\rm core}(\ell-\ell_i) + \omega(\ell), \qquad \forall \ell \in \L_N,
\end{eqnarray}
where the defect core truncation operator $\Pi_R$ is defined by \eqref{eq:trun_op_def} with radius $R=L_{\D}/3$, and where $\omega \in \Admu$ satisfies
\begin{eqnarray}
\|D\omega\|_{\ell^2_{\mathcal{N}}(\L_N)} \leq C \sqrt{n_{\D}} \cdot L_{\D}^{-d/2},
\end{eqnarray}
with a constant $C$ independent of $N, n_{\D}, L_{\D}$. 
\end{theorem}

The condition $L_{\D} \geq L_0$ entails that defect cores do not overlap too strongly. Theorem~\ref{them:existence} not only gives the existence of the equilibrium of the multiple point defects in a periodic domain, but also establishes its structure: The equilibrium can be decomposed into two parts, a truncated defect core centered at each point defect and a remainder term. Our generalisation analysis in the next section heavily relies on this result.

\section{Main results}
\label{sec:sub:main}

\subsection{Error estimates}
\label{sec:sub:error}

The computational cost of solving \eqref{eq:variational-problem-per} on a large simulation domain is prohibitive when an electronic structure model is taken into consideration. Surrogate models are therefore introduced, fitted to an electronic structure model, to give reasonable approximations.
If we denote the surrogate site potential by $\widetilde{V}_{\ell}$ (resp. $\widetilde{V}^{\rm h}$), then the corresponding energy-difference functionals are given by
\begin{eqnarray}
\label{energy-difference-approx}
\widetilde{\E}(u) := \sum_{\ell\in\L_N} \widetilde{V}_{\ell}\big(Du(\ell)\big)
\qquad \text{and} \qquad 
\widetilde{\E}^{\rm h}(u) := \sum_{\ell\in\Lhom} \widetilde{V}^{\rm h}\big(Du(\ell)\big).
\end{eqnarray}
The resulting variational problem for the equilibration reads
\begin{equation}\label{eq:variational-problem-approx}
\tilde{u} \in \arg\min \big\{ \widetilde{\E}(u), u \in \Admu \big\}.
\end{equation}

The surrogate model enables large-scale simulation but obtaining ab initio training data on a large simulation domain remains intractable. Instead, one normally fits the parameters in the surrogate model to ab initio simulations on very small training domains containing at most a few hundred atoms. These should include all possible local snapshots one expects to encounter in the simulation; in our case, regions of homogeneous crystal or single defects. In the following we will give a rigorous framework within which this intuition is made precise and all the resulting errors are quantified.

To that end, given $L_0\leq L\leq L_{\D}$, we call $\Omega_L:=\mathsf{B}(-L/2, L/2]^d$ the {\it training domain} while $\Omega_N$ is called the {\it simulation domain}. Let $\L_L:= \L \cap \Omega_L$ contain a single defect core located at the origin. Similarly as in the previous section, let $\Us^{\rm per}_{L}$ and $\Admu_L$ be the corresponding space of periodic displacements and admissible set. We equip $\Us_L^{\rm per}$ with the norm $\|u\|_{\Us_L^{\rm per}} := \|Du\|_{\ell^2_{\mathcal{N}}(\L_L)}$. Let $\E_L(u)$ and $\widetilde{\E}_L(u)$ be the energy functionals defined on $\L_{L}$, and the equilibrium of the corresponding variational problem with $\E_L(u)$ is denoted as $\bar{u}_L$.

Next, we introduce the {\it matching conditions} between the reference and the {\it approximated} models, in terms of the observations (energies, forces and force constants) of the possible configurations near $\bar{u}_L$. To be more precise, let $\delta>0$ such that $B_{\delta}(\bar{u}_L) \subset \Admu_L$ representing atomic displacements near $\bar{u}_L$, then the {\it matching condition} for the energy is defined by
\begin{eqnarray}\label{eq:ED}
\varepsilon^{\rm E} := \max_{u_L \in B_{\delta}(\bar{u}_L)} \big| \E_L(u_{L}) - \widetilde{\E}_L(u_{L})\big|.
\end{eqnarray}
To measure the force error, we introduce
\begin{eqnarray}\label{eq:ffit}
\varepsilon^{\rm F} := \max_{u_L \in B_{\delta}(\bar{u}_L)} \big\|-\nabla \E_L(u_{L})+ \nabla \widetilde{\E}_L(u_{L})\big\|_{(\Us^{\rm per}_{L})^{*}}, 
\end{eqnarray}
where $\|\cdot\|_{(\Us^{\rm per}_{L})^{*}}$ is the dual norm of $\Us^{\rm per}_{L}$.
Furthermore, for the stability analysis, we require the force constant error 
\begin{eqnarray}\label{eq:FCD}
\varepsilon^{\rm FC} := \big\|\nabla^2 \E_L(\bar{u}_{L}) - \nabla^2 \widetilde{\E}_L(\bar{u}_{L})\big\|_{\mathcal{L}(\Us^{\rm per}_{L}, (\Us^{\rm per}_{L})^{*})}. 
\end{eqnarray}
In particular, the force constant error on the homogeneous lattice is given by
\begin{eqnarray}\label{eq:FCDhom}
\varepsilon^{\rm FC}_{\rm hom} := \big\|\nabla ^2\mathcal{E}^{\rm h}({\bm 0}) - \nabla^2 \widetilde{\mathcal{E}}^{\rm h}({\bm 0})\big\|_{\mathcal{L}(\Us^{\rm per}_{L}, (\Us^{\rm per}_{L})^{*})}.
\end{eqnarray}

The following result provides a rigorous {\it a priori} error estimate of the geometry error $\|D\bar{u}-D\tilde{u}\|_{\ell^2_{\mathcal{N}}(\L_N)}$ and of the formation energy error $|\E(\bar{u}) - \widetilde{\E}(\tilde{u})|$ in terms of the size of the training domains and in terms of the fit accuracy of the surrogate model. 
The proof is given in Section \ref{sec:sub:thm2}.

\begin{theorem}\label{them:geometry}
Suppose that \asS~is satisfied and that $\bar{u}$ is a strongly stable equilibrium of \eqref{eq:variational-problem-per}. Then, for $L, L_{\D}$ sufficiently large, satisfying $L\leq L_{\D}$, and $\varepsilon^{\rm FC}$ and $\varepsilon^{\rm FC}_{\rm hom}$ sufficiently small, there exists an equilibrium $\tilde{u}$ of the surrogate model \eqref{eq:variational-problem-approx} such that
\begin{align}\label{eq:geoerr}
\|D\bar{u} - D\tilde{u}\|_{\ell^2_{\mathcal{N}}(\L_N)} &\leq~ C^{\rm G} \sqrt{n_{\D}} \cdot \big( \varepsilon^{\rm F} + L^{-d/2} \varepsilon^{\rm FC}_{\rm hom} + L^{-3d/2}\big), \\
\label{eq:ergyerr}
\big| \E(\bar{u}) - \widetilde{\E}(\tilde{u}) \big| &\leq~ C^{\rm E} \,\, n_{\D} \cdot \Big( \big( \varepsilon^{\rm F} + L^{-d/2} \varepsilon^{\rm FC}_{\rm hom} + L^{-3d/2}\big)^2 + L^{-d} + \varepsilon^{\rm E}  \Big),
\end{align}
where both constants $C^{\rm G}$ and $C^{\rm E}$ are independent of $N, n_{\D}, L$.
\end{theorem}

The error estimates in the foregoing theorem identify how the geometry error and the error in formation energy depend on data-oriented approximation parameters: model accuracy on the {\it training domain} and its size, $L$. If we construct the {\it approximated} site potential $\widetilde{V}_{\ell}$ such that the {\it matching conditions} (from \eqref{eq:ED} to \eqref{eq:FCDhom}) are exactly zero, we obtain rates of convergence in terms of $L$. Conversely, if $L$ is sufficiently large, the errors then depend only on the {\it matching conditions} $\varepsilon^{\rm F}, \varepsilon^{\rm E}$. These limiting cases are summarized in Table~\ref{table-e-mix}. We will see in Section~\ref{sec:sub:numerics} that these rates are indeed sharp. 


\begin{table}
	\begin{center}
	\vskip0.2cm
	\begin{tabular}{|c|c|c|c|}
	\hline
		Errors ($d=2,3$) & {$\varepsilon^{\rm E, F}=0$} & {$\varepsilon^{\rm E, F}=\varepsilon^{\rm FC}_{\rm hom}=0$} & {$L$ sufficiently large} \\[1mm]
		\hline
		Geometry &  $L^{-d/2}$ & $L^{-3d/2}$ & $\varepsilon^{\rm F}$
		 \\[1mm]
		 \hline
		Energy &  $L^{-d}$ & $L^{-d}$ & $ \varepsilon^{\rm E} + \big(\varepsilon^{\rm F}\big)^2$
		\\[1mm]
		\hline 
	\end{tabular}
	\medskip 
	\caption{Limiting cases of error decay with respect to $L$ and the matching conditions $\varepsilon^{\rm E}, \varepsilon^{\rm F}$ and $\varepsilon^{\rm FC}_{\rm hom}$.}
	\label{table-e-mix}
	\end{center}
\end{table}

There are several insights we can gain from our results. First, they clearly tease out an issue that is --- to the best of our knowledge --- never discussed in the MLIPs literature: the size of training domains significantly affects the quality of the fitted model. Intuitively this happens because one cannot in practise obtain ``perfect'' snapshots. Secondly, we see the importance of fitting force constants in order to significantly reduce the effect that size of the simulation domain has. Finally, our estimates provide a clear guidance on how energy, force and force-constant observations should be weighted in the least squares loss function, in particular suggesting the optimal balance $\efit \approx (\ffit)^2$ suggesting to put signifantly higher emphasis on the energy fit. 


\section{Numerical results}
\label{sec:numer}
We propose a concrete implementation of MLIPs inspired by our generalisation analysis of Theorem \ref{them:geometry} to confirm the analytical error estimates on model problems.

\subsection{Constructions of MLIPs}
\label{sec:sub:consACE}
%

\subsubsection{Parameterisation}
\label{sec:sub:sub:para}
First, we need to choose a parameterisation of the surrogate potential $\widetilde{V}$. Although a wide variety of choices is available nowadays, we have opted for the {\it linear} atomic cluster expansion (ACE)~\cite{2019-ship1, lysogorskiy2021performant, oord19} which has the advantage of achieving close to state of the art accuracy despite being a linear model~\cite{lysogorskiy2021performant}. However, there is no reason to believe that the specific choice of MLIP is essential in our tests. 
Briefly, in the ACE model, the surrogate potential $\tilde{V}$ is written as 
\begin{eqnarray}\label{eq:ace}
	\widetilde{{V}}(\pmb{g}) := V^{\rm ACE}(\pmb{g}; \{c_B\}_{B\in\pmb{B}}) = \sum_{B\in\pmb{B}} c_B B(\pmb{g}), 
\end{eqnarray}
%
where $B$ are the ACE basis functions and $\{c_B\}_{B\in\pmb{B}}$ are the parameters that we will estimate by minimizing a least squares loss. The basis functions $B$ are invariant under rotations, reflections and permutations of an atomic environment. 

A more detailed review of the ACE model and in particular its approximation parameters is provided in the Appendix~\ref{sec:ACE}. 


\subsubsection{Training sets and loss}
\label{sec:sub:sub:train}

Following our generalisation analysis (Theorem \ref{them:geometry}) we require that the ACE model matches the reference model in the sense of making $\efit, \ffit, \varepsilon^{\rm FC}$ and $\fcfit$ small. These measures of fit accuracy are specified in terms of max-norms over an infinite set of displacements, which is clearly computationally not tractable. At this point we make several departures from our rigorous analysis. 

We first introduce the training set, $\mathfrak{R}$: The complete neighbourhood $B_{\delta}(\bar{u}_L)$ used in the analysis is replaced with a finite number of random samples taken from $B_{\delta}(\bar{u}_L)$. We fix the perturbed parameter to be $0.01$ and denote the number of the configurations in $\mathfrak{R}$ as $N_{\rm train}:= \#\mathfrak{R}$. Analogously we also produce a test set. The number of configurations in training and test sets will be specified for each individual example.


Next, we consider the construction of a loss function inspired by our theory. We cannot optimize 
\[
\efit + \ffit + \varepsilon^{\rm FC} + \fcfit
\]
directly but we propose three {\it ad hoc} approximations of the {\it matching conditions}:
%
\begin{itemize}
    \item We replace the max-norm with an $\ell^2$-norm to obtain a linear least squares problem. 

    \item The force error $\ffit$ is defined in a dual norm in \eqref{eq:ffit}, but for sufficiently small training domains it is almost indistinguishable from a standard $\ell^2$-norm, namely $\big\|-\nabla \E_L(u_{L})+ \nabla \widetilde{\E}_L(u_{L})\big\|_{\ell^2}$.
    
    \item Finally, we drop the the force constant errors $\varepsilon^{\rm FC}$ and $\fcfit$ entirely from the loss function. In fact, we have found that constructing the training set $\mathfrak{R}$ as described above and only fitting forces and energies already results in a sufficiently good accuracy of $\varepsilon^{\rm FC}$. We will give a numerical verification of this statement in Section~\ref{sec:sub:numerics}, Table~\ref{table-params}.
	 %
	%
\end{itemize}


Given the foregoing approximations, the training set $\mathfrak{R}$ constructed above and the parameterisation defined by \eqref{eq:ace}, we determine the parameters $\{c_B\}$ by minimising the following loss function
\begin{align}\label{cost:energymix}
\mathcal{L}\big(\{c_B\}\big) := \sum_{u_R \in \mathfrak{R}}\Big(W_{\rm  E} \big|\E_L(u_R) - \widetilde{\E}_L(u_R; \{c_B\})\big|^2 + W_{\rm F}\big|\nabla \E_L(u_R) - \nabla \widetilde{\E}_L(u_R; \{c_B\}) \big|^2 \Big),
\end{align}
where $W_{\rm E}$ and $W_{\rm F}$ are additional weights that might depend on the configurations and observations. According to \eqref{eq:ergyerr} (or Table \ref{table-e-mix}), we choose $W_{E} \gg W_{F}$ in practice such that the balance $\efit \approx (\ffit)^2$ can be achieved. The details will be provided for different model problems in the next section. 

The loss function \eqref{cost:energymix} is quadratic in the parameters $\{c_B\}$ and can therefore be minimised by using a QR factorisation. In our implementation we use a rank-revealing QR (rr-QR) factorisation \cite{chan1987rank} which provides a mechanism analogous to Tychonov-regularisation~\cite{oord19}. The regularisation parameter for rr-QR factorisation is set to be $10^{-6}$ (we tested various values and it gave the best performance for all numerical experiments) throughout this work.


\subsection{Numerical results}
\label{sec:sub:numerics}

We present numerical tests for two examples: 
\begin{enumerate} 
\item {\it Toy model: } We consider various configurations of point defects (vacancies and interstitials) in a two dimensional triangular lattice. As the reference model we will use an embedded atom model (EAM) \cite{Daw1984a} instead of an electronic structure model. This highly simplified scenario, and the fact that elastic fields decay more slowly in two dimensions, allows us to more easily perform large-scale simulation in which we can most clearly observe the expected convergence results.

\item {\it Tight-Binding model: } We will also perform tests on configurations of multiple vacancies in three dimensional Silicon (Si). For these tests we employ the NRL tight binding model \cite{cohen94,mehl96,papaconstantopoulos15} as the reference model; see Appendix \ref{sec:NRL} for a brief review.
\end{enumerate} 

All numerical tests are implemented in open-source {\tt Julia} packages {\tt ACE1.jl} \cite{gitACE} (for the ACE model) and {\tt SKTB.jl} \cite{gitSKTB} (for the NRL tight binding model). 


\subsubsection{Two-dimensional toy model.}
\label{sec:sub:eamW}
To demonstrate the main theory (Theorem~\ref{them:geometry}) most clearly, we first explore two-dimensional multiple point defects systems where the reference model is given by an EAM potential \cite{Daw1984a}. 
We consider a two dimensional triangular lattice 
\begin{eqnarray}\label{eq:mA}
\L^{\rm h}:=\mA \Z^2, \quad \textrm{where} \quad \mA=r_0 \cdot\Bigg(\begin{array}{cc}
    1 & 1/2 \\
    0 & \sqrt{3}/2
\end{array} \Bigg),
\end{eqnarray}
with $r_0$ chosen such that the triangular lattice becomes the ground state of the reference EAM potential.
%
Using the triangular lattice and an EAM potential as the reference model instead of an actual ab initio model means that we can more easily perform large-scale tests in a wider parameter regime in order to narrow down the best choices. In this example, the size of the simulation domain $\L_N$ is chosen to be $N=60r_0$. The separation distance $L_{\D}$ is identical to the size of training domain $L$ for the sake of simplicity.

{\it Multi-vacancies:} We first consider the multi-vacancies case. The corresponding simulation domain and training domain for two separated vacancies ($n_{\D}=2$) are shown in Figure \ref{fig:multivac}. The illustration of other two cases ($n_{\D}=3,4$) considered in this example is provided in Figure~\ref{fig:multivac_app} in the Appendix~\ref{sec:numer_supp}.

\begin{figure}[!htb]
\centering
	\includegraphics[height=5.0cm]{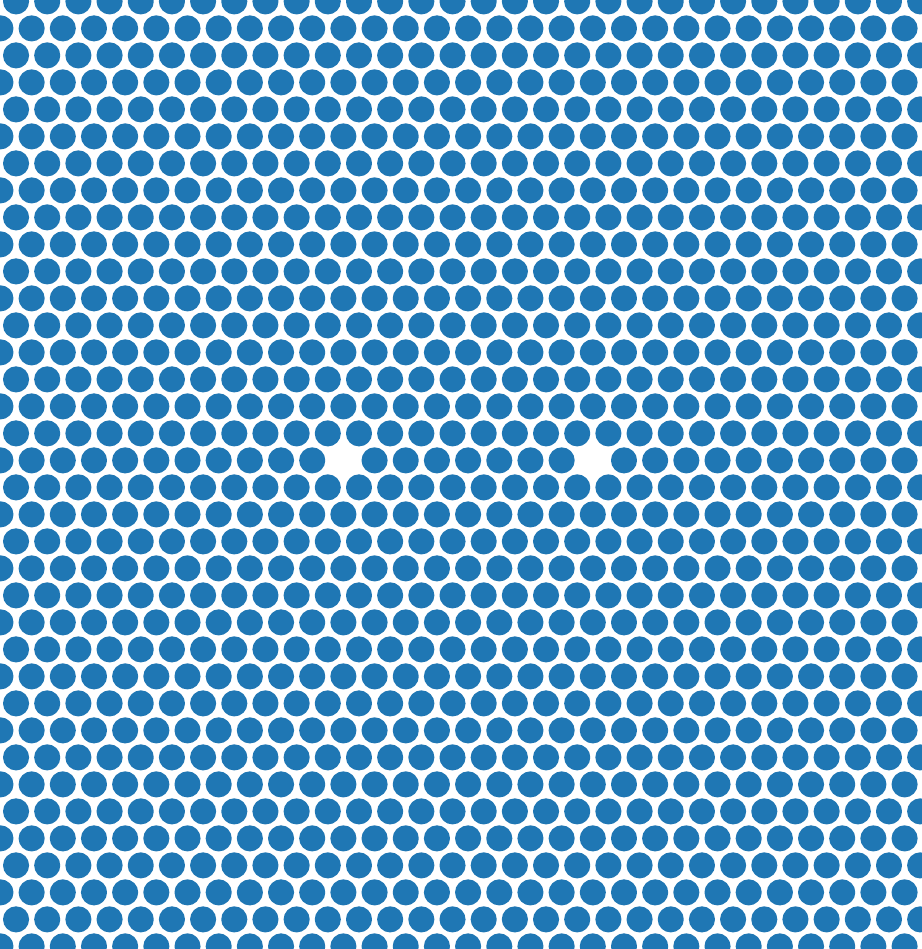}
	\qquad \quad 
	\includegraphics[height=3.8cm]{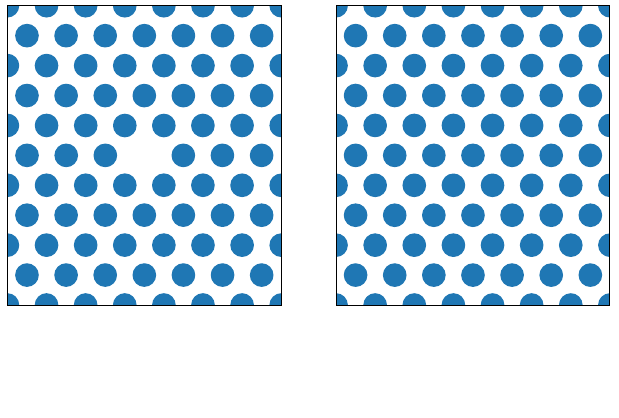}
    \caption{EAM W reference model: Illustration of the simulation domain (left) and the training domains (right) for two separated vacancies with $L=L_{\D}=8r_0$.}
    \label{fig:multivac}
\end{figure}

The MLIPs are fitted by following the construction in Section~\ref{sec:sub:consACE}, where the parameters in establishing the basis functions $B$ are taken from \cite[Section 7.5]{2019-ship1}. The number of configurations in training and testing sets are set to be $N_{\rm train}=200$ and $N_{\rm test}=50$, respectively. We choose the additional weights for energy and force in \eqref{cost:energymix} as $W_{\rm E}=100$ and $W_{\rm F}=1$ in order to balance the matching conditions $\efit \approx (\ffit)^2$ (cf. Table~\ref{table-e-mix}). The fitting parameters, fitting accuracy in terms of the root mean square error (RMSE) on the testing sets and fitting time for various size of training domain $L$ are given in Table \ref{table-params} in the Appendix~\ref{sec:numer_supp}. The accuracy of the force constant on the homogeneous lattice is also verified in Table \ref{table-params}. As we discussed in Section~\ref{sec:sub:sub:train}, the force constants are already fitted well (relative RMSE are about $5\%$) despite only fitting forces and energies.

We first test the convergence of the geometry error $\|D\bar{u} - D\tilde{u}\|_{\ell^2_{\mathcal{N}}}$ and the error in energy $|E(\bar{u})-\widetilde{E}(\tilde{u})|$ with respect to the RMSE by studying the case of two separated vacancies (Figure \ref{fig:multivac}). Figure \ref{fig:errorvsrmse} shows that, for different size of training domain $L$, the error curves of geometry error and error in energy decrease near linearly and quadratically respectively as RMSE decreases, which perfectly matches our theoretical predictions from Theorem \ref{them:geometry}.

\begin{figure}[!htb]
    \centering
    \includegraphics[height=5.5cm]{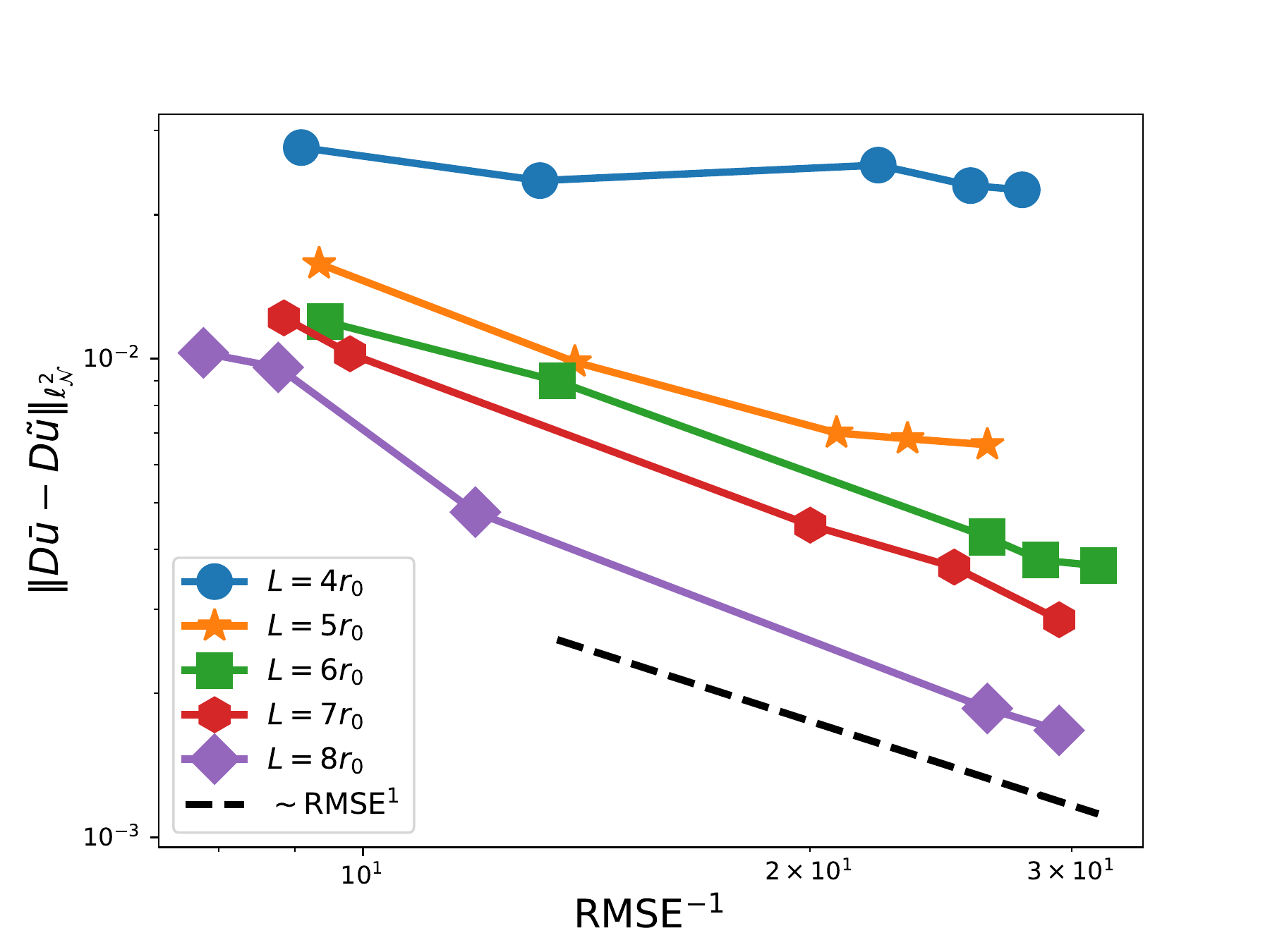}
    \includegraphics[height=5.5cm]{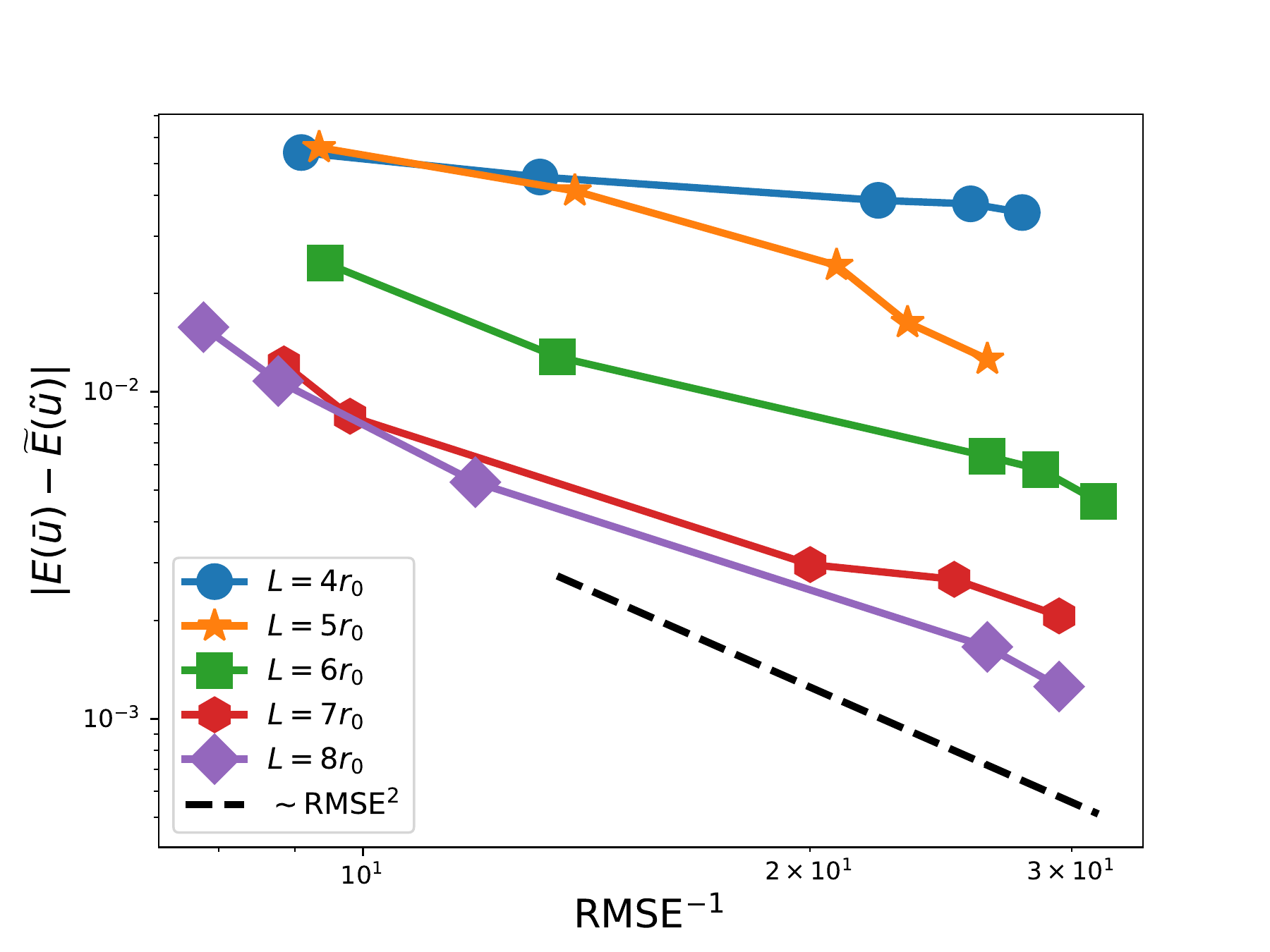}
    \caption{2D toy model: Geometry error (left) and error in energy (right) v.s. RMSE for two separated vacancies case.}
    \label{fig:errorvsrmse}
\end{figure}

Figure \ref{fig:tri} plots the convergence of the geometry error $\|D\bar{u} - D\tilde{u}\|_{\ell^2_{\mathcal{N}}}$ and the error in energy $|E(\bar{u})-\widetilde{E}(\tilde{u})|$ against the size of training domain $L$. We observe that the convergence rates of geometry error perfectly match our theoretical predictions from Theorem~\ref{them:geometry} for all multi-vacancies cases. The only exception is the convergence of error in energy, a faster convergence rate is observed numerically. We speculate that this is due to the additional symmetry of the defective lattice for the cases we consider here, and that this leads to additional cancellation that our general analysis does not capture. 

\begin{figure}
	\centering 
		\includegraphics[height=5.5cm]{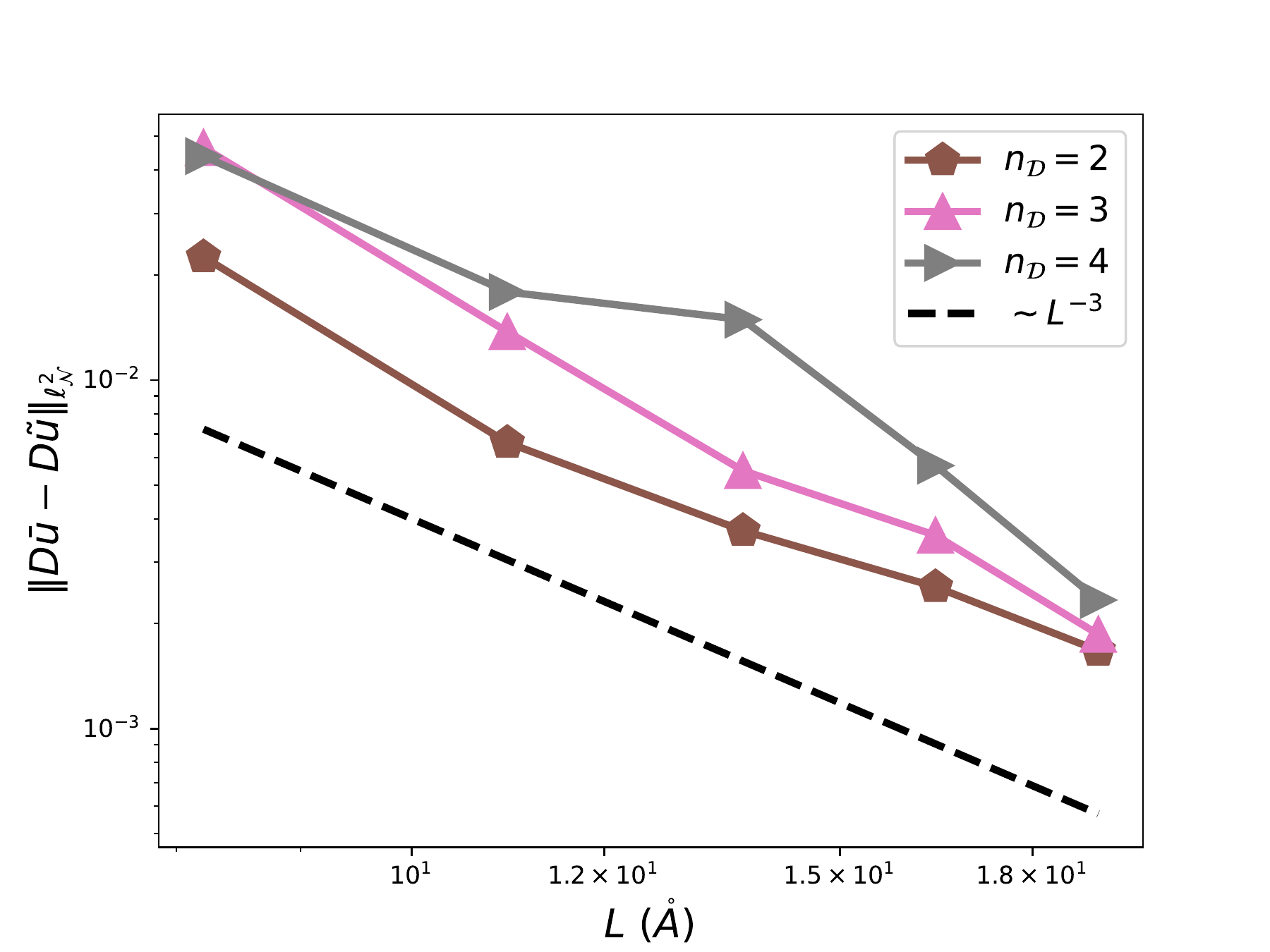}
		\includegraphics[height=5.5cm]{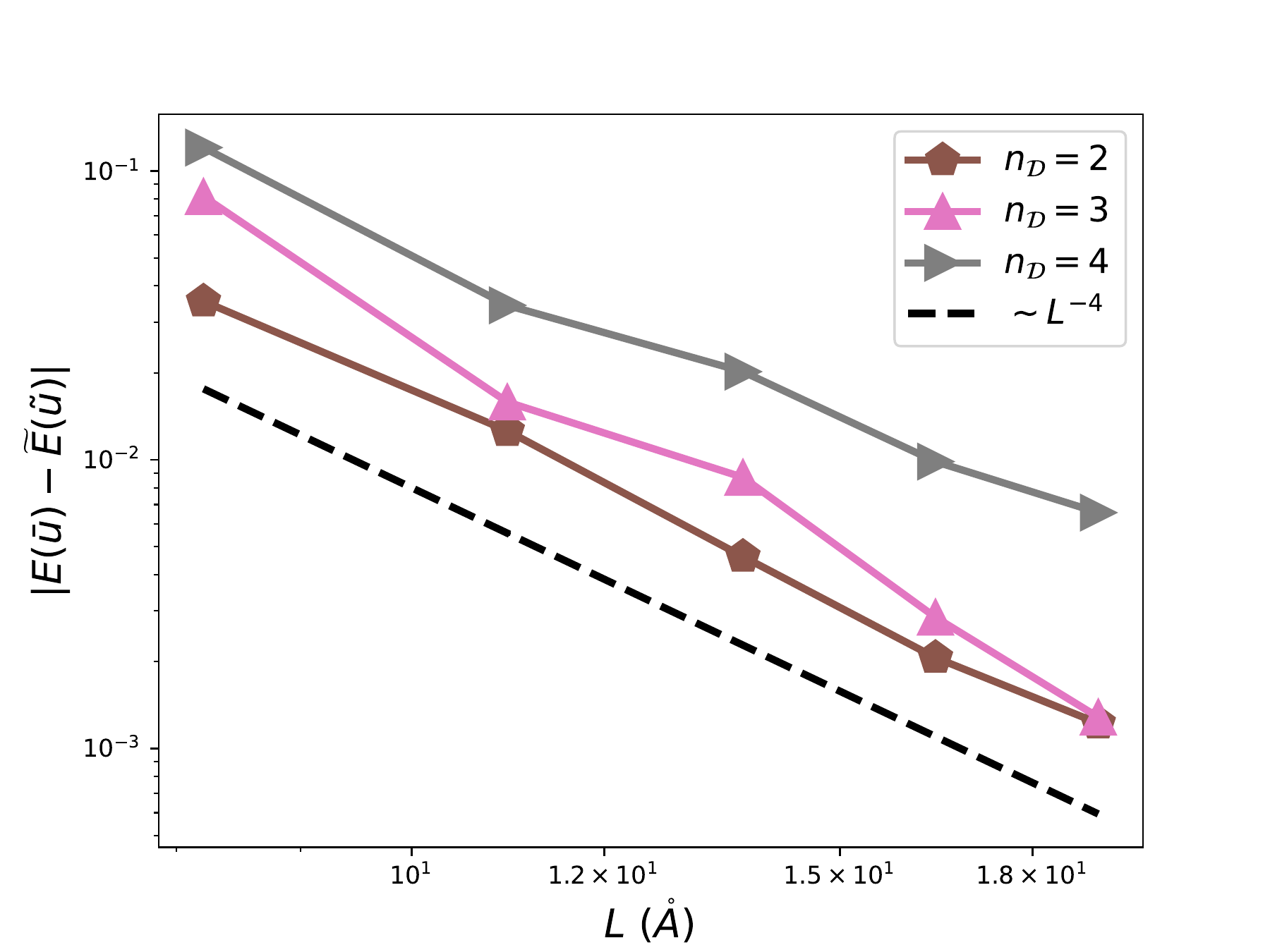}
	\caption{2D toy model: Convergence rates for geometry error (left) and error in energy (right) for multi-vacancies, ${\rm RMSE}\approx 0.03$.}
	\label{fig:tri}
\end{figure}

{\it Interstitials-vacancies:} Next, we test the interstitials-vacancies case. Figure~\ref{fig:intvac} plots the corresponding simulation domain and training domain for one interstitial and one vacancy ($n_{\D}=2$). The case of one interstitial and three vacancies ($n_{\D}=4$) illustrated in Figure~\ref{fig:intvac_app} in the Appendix~\ref{sec:numer_supp} will also be considered in this example.

\begin{figure}
    \centering
	\includegraphics[height=5.0cm]{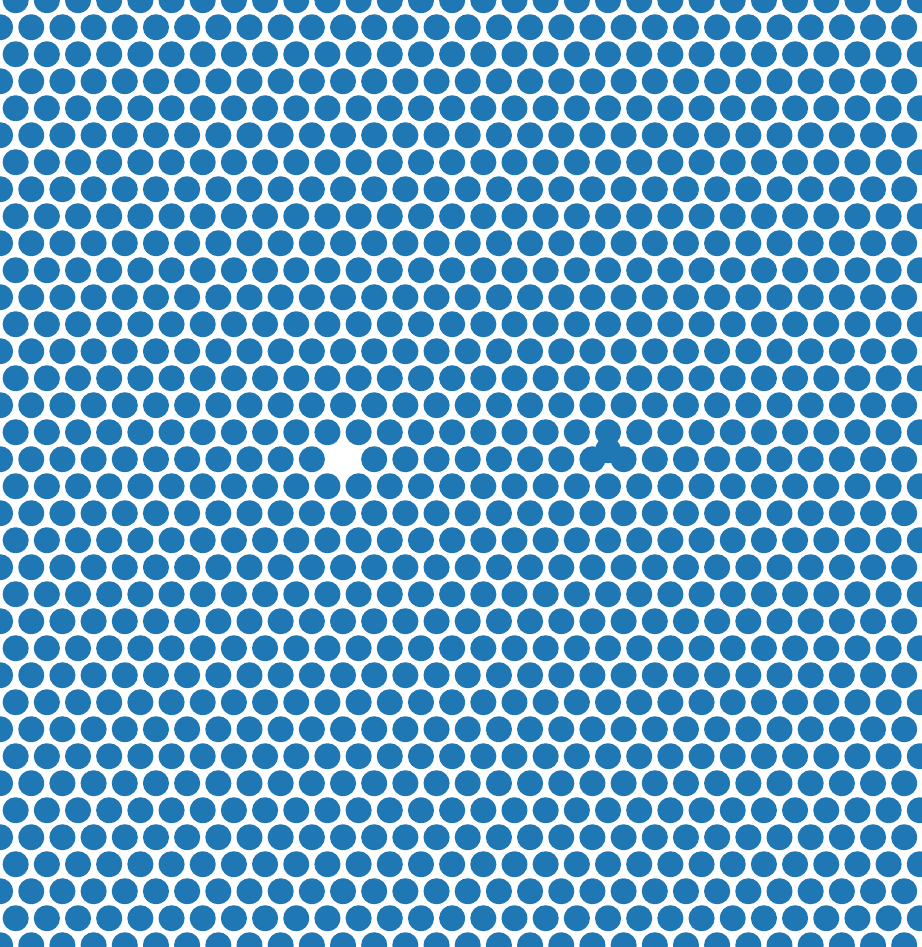} 
	\qquad 
	\includegraphics[height=3.8cm]{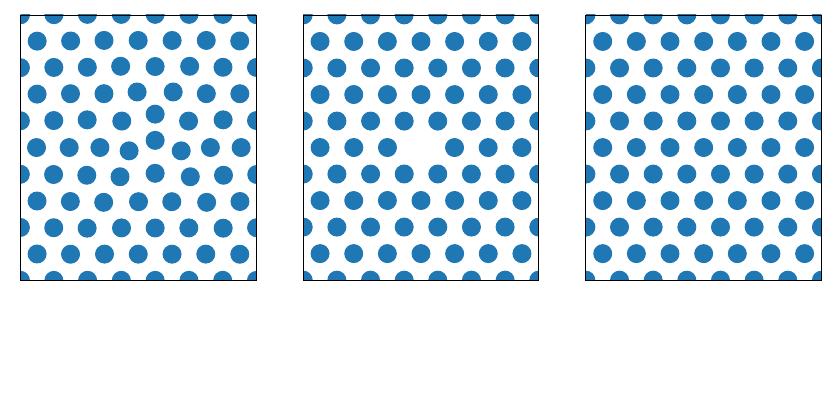}
    \caption{2D toy model: Illustration of the simulation domain (left) and the training domains (right) for one interstitial and one vacancy. }
    \label{fig:intvac}
\end{figure}

We take $N^{\rm vac}_{\rm train}=N^{\rm int}_{\rm train}=100$, and $N^{\rm vac}_{\rm test}=N^{\rm int}_{\rm test}=25$ for both cases. To balance the errors ($\efit \approx (\ffit)^2$), the additional weights in \eqref{cost:energymix} are chosen to be: $W^{\rm int}_{\rm E}=100$ and $W^{\rm int}_{\rm F}=10$ for interstitial while $W^{\rm vac}_{\rm E}=10$ and $W^{\rm vac}_{\rm F}=1$ for vacancy. Here we put more weights on the interstitial to express the fact that the it generates a larger distortion of the surrounding lattice thus making it more challenging to fit than the vacancy cores. 

As presented in the previous case, we take $n_{\D}=2$ (cf. Figure~\ref{fig:intvac}) as an example to study the convergence of geometry error and error in energy with respect to the RMSE on testing sets. Similar to the multi-vacancies case, Figure~\ref{fig:errorvsrmse:int} shows that the errors are reduced as the accuracy of corresponding ACE models is improved. 

\begin{figure}
    \centering
    \includegraphics[height=5.5cm]{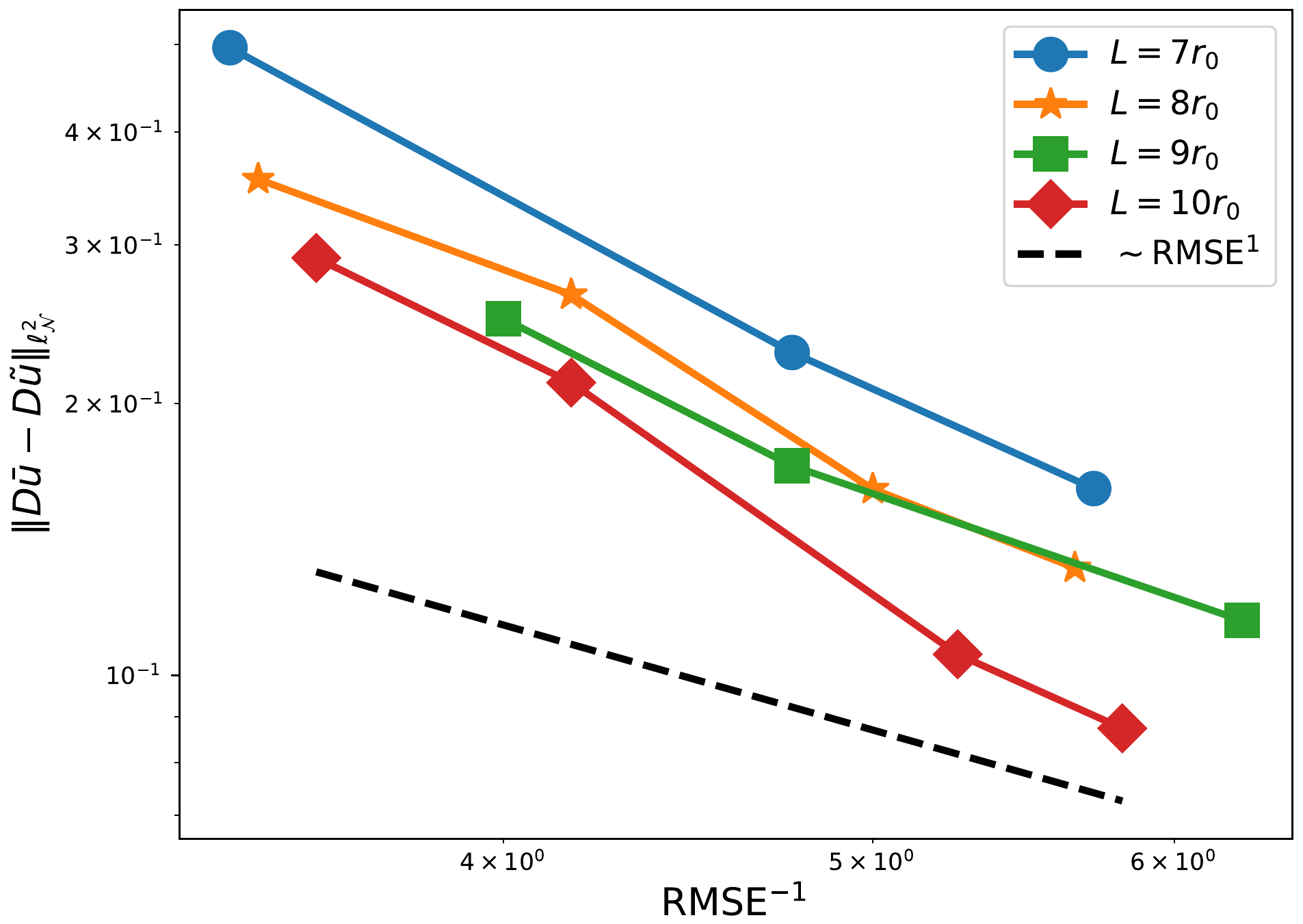}
    \includegraphics[height=5.5cm]{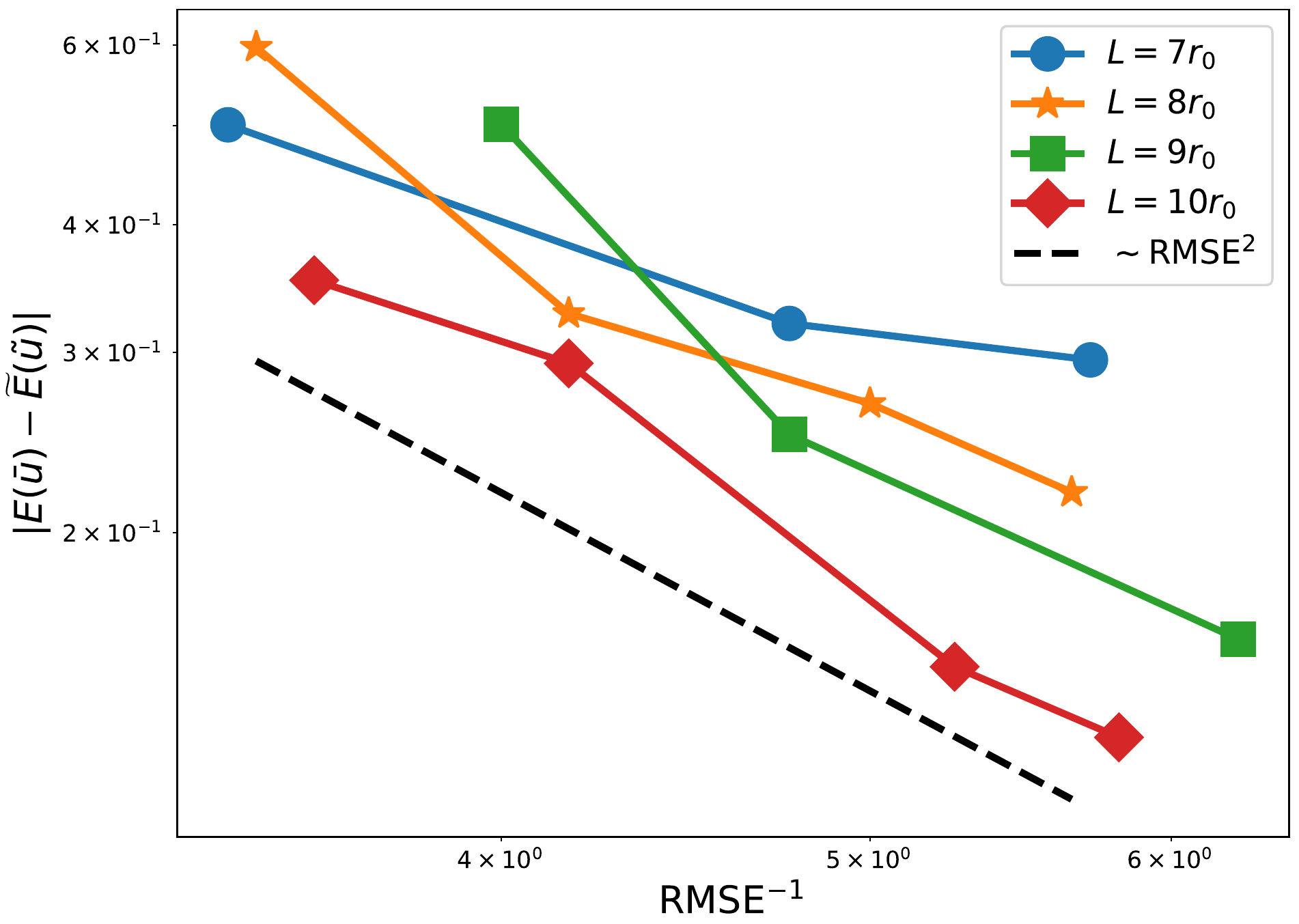}
    \caption{2D toy model: Geometry error (left) and error in energy (right) v.s. RMSE for one interstitial and one vacancy case.  }
    \label{fig:errorvsrmse:int}
\end{figure}

The decay of geometry error $\|D\bar{u} - D\tilde{u}\|_{\ell^2_{\mathcal{N}}}$ and corresponding error in energy $|E(\bar{u})-\widetilde{E}(\tilde{u})|$ against the size of training domain $L$ are shown in Figure~\ref{fig:tri_intvac}. We observe that the convergence rates roughly match our theoretical predictions from Theorem~\ref{them:geometry} which are summarized in Table~\ref{table-e-mix}.

\begin{figure}
	\centering 
		\includegraphics[height=5.5cm]{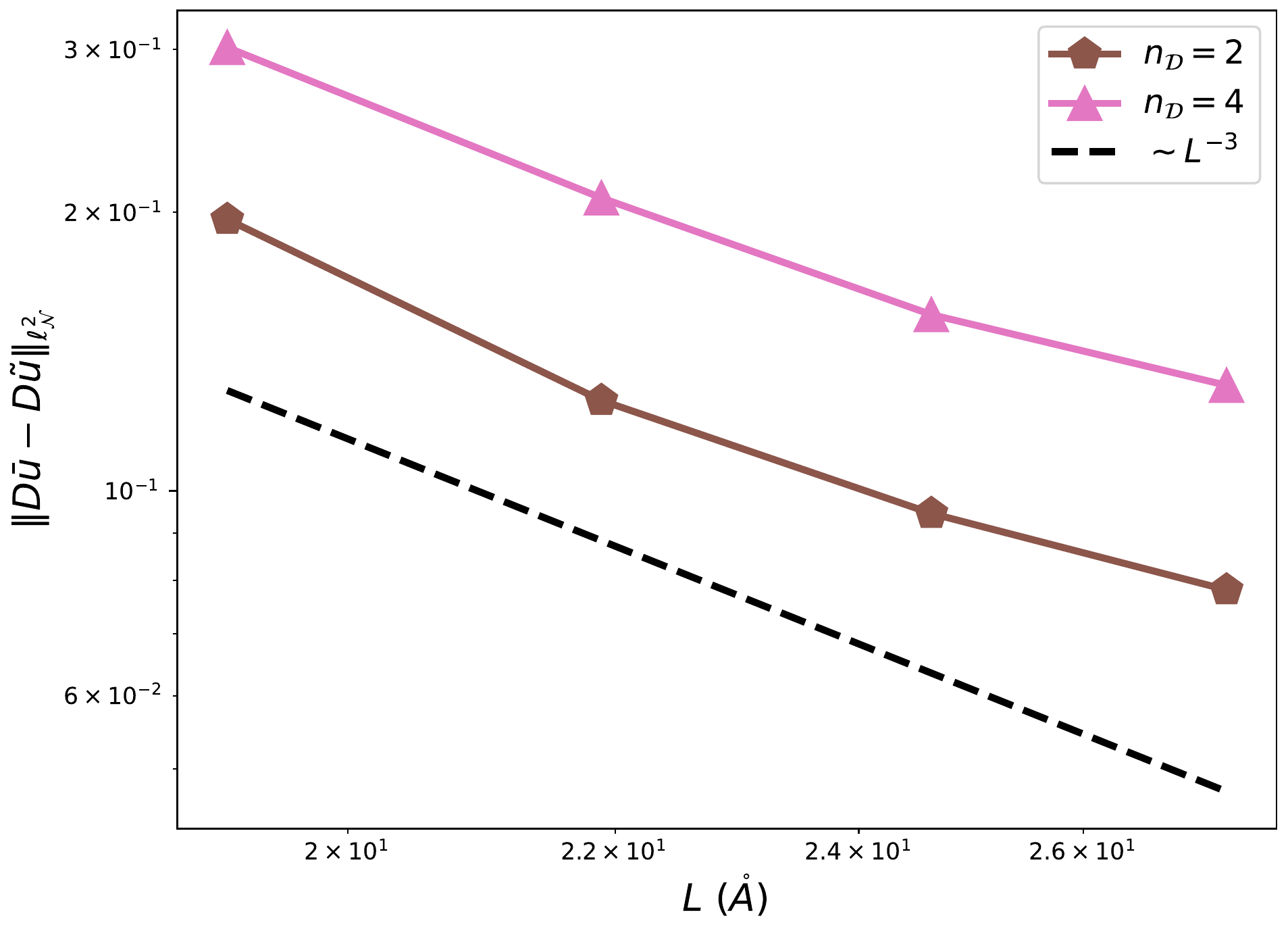}
		\includegraphics[height=5.5cm]{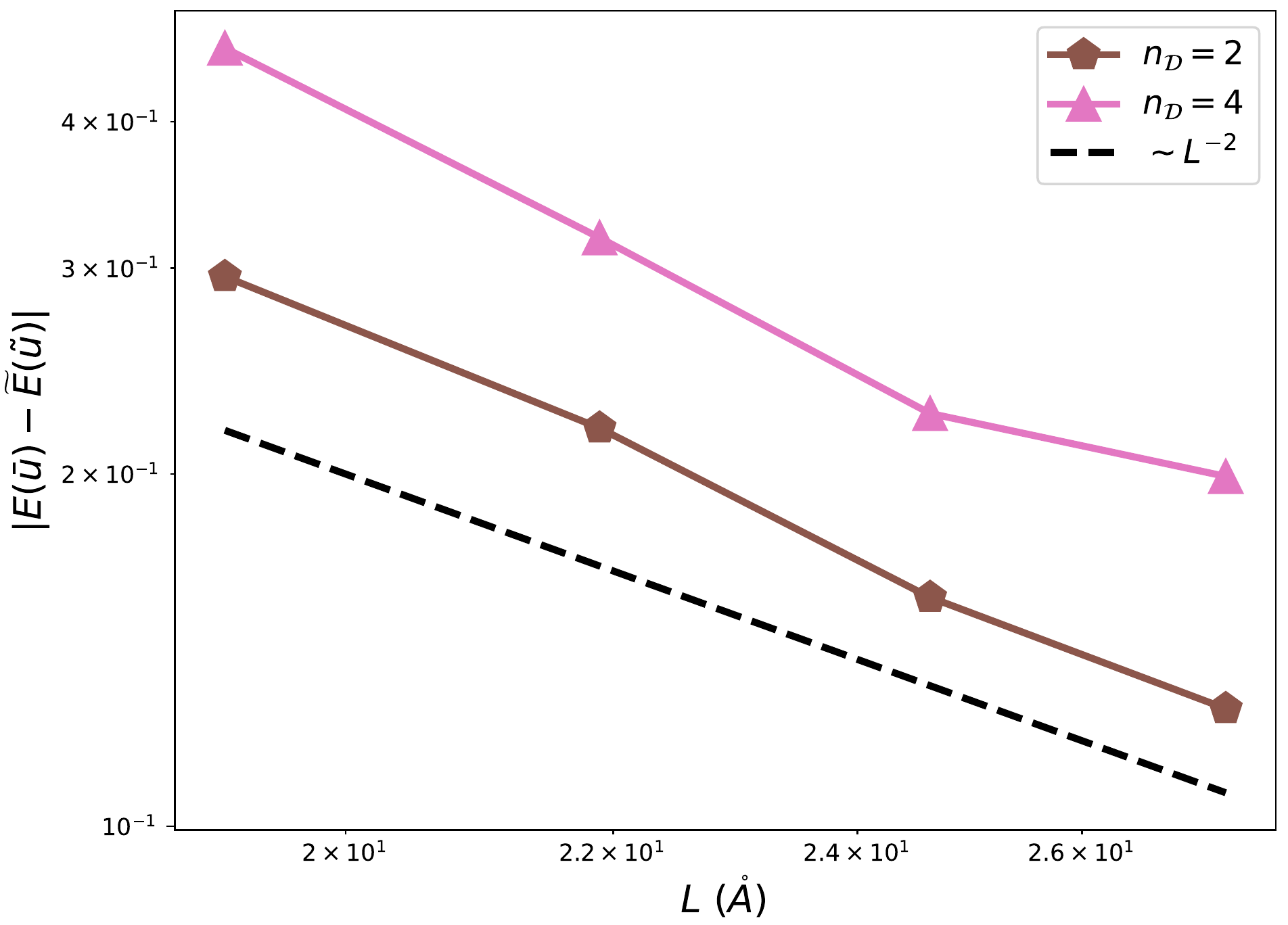}
	\caption{2D toy model: Convergence rates for geometry error (left) and error in energy (right) for interstitials-vacancies, ${\rm RMSE}\approx 0.15$.}
	\label{fig:tri_intvac}
\end{figure} 

\subsubsection{Convergence results for the NRL-TB model.}
\label{sec:sub:NRLTB}

We now move to testing our schemes when the reference model is an electronic structure model. We choose NRL-TB \cite{cohen94} as the reference model, which is a successful tight-binding model for Si; see Appendix \ref{sec:NRL} for a short review. Our choice of Si as the material is due to the fact that it is a semi-conducting material for which we have also strong theoretical and numerical evidence for the localisation of its interatomic forces~\cite{chen18}, which is an essential ingredient in our analysis. 

The simulation domain is constructed by $10^3$ unit cells, which contains 8000 Si atoms with periodic boundary conditions in all three directions. We select two vacancy sites which are separated from each other in Silicon bulk crystal. The MLIPs are fitted by following the construction in Section \ref{sec:sub:consACE}, where the parameters in building the basis functions $B$ for Si are taken from \cite[Section 7.5]{2019-ship1}. The total number of the configurations in the training and testing sets and the additional weights in \eqref{cost:energymix} are chosen to be the same as that in multi-vacancies presented in the previous section.

The convergence of geometry error and energy error against RMSE for NRL-TB Si model is shown in Figure~\ref{fig:errorvsrmse:3D}, where the predicted convergence is again observed for this electronic structure model. Figure \ref{fig:3D_vac} plots the decay of geometry error $\|D\bar{u} - D\tilde{u}\|_{\ell^2_{\mathcal{N}}}$ and corresponding error in energy $|E(\bar{u})-\widetilde{E}(\tilde{u})|$ against the size of training domain $L$. We observe that the convergence rates again perfectly match our theoretical predictions from Theorem~\ref{them:geometry} and Table \ref{table-e-mix} for NRL-TB Si reference model.

\begin{figure}
    \centering
    \includegraphics[height=5.5cm]{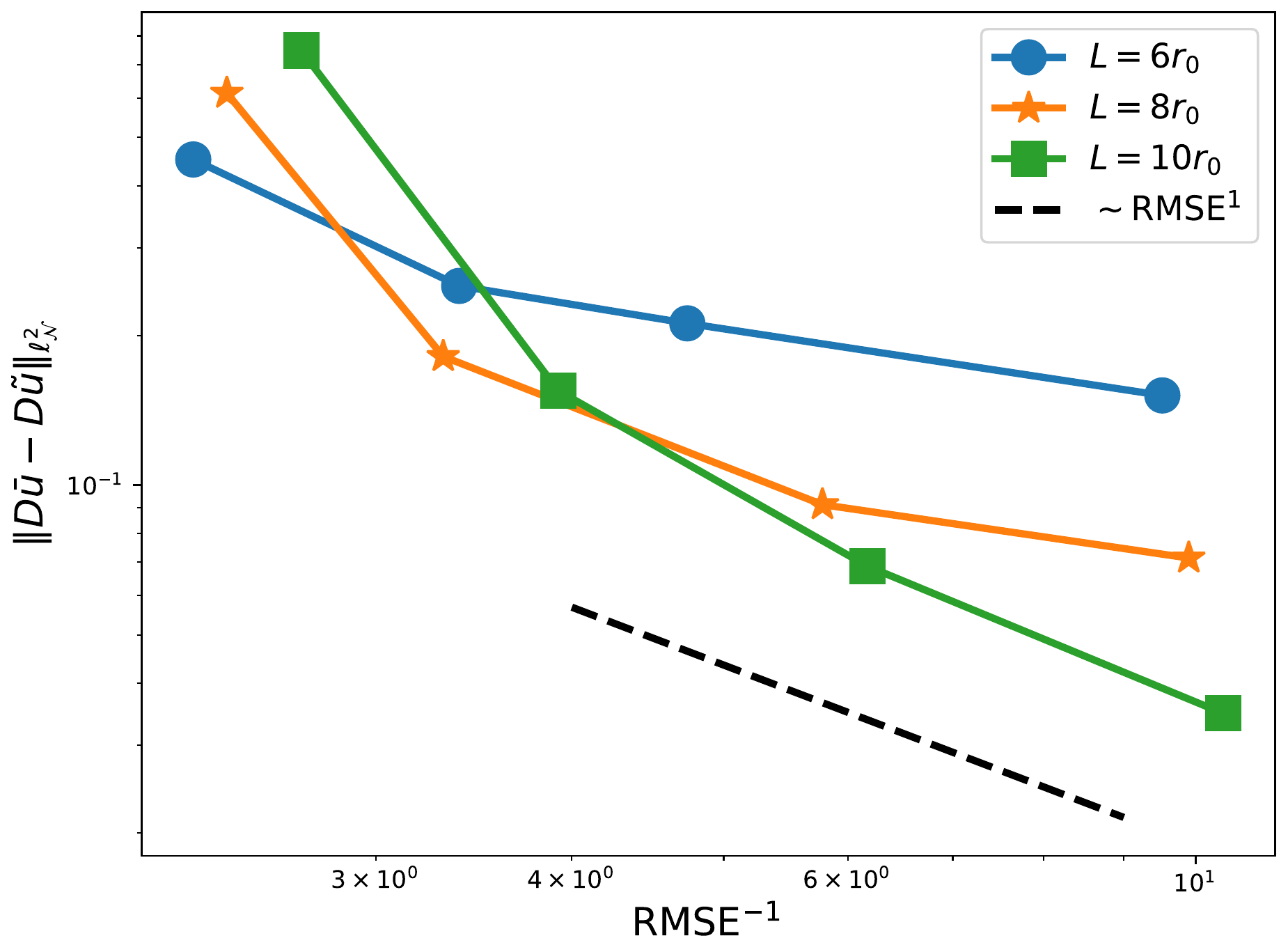}\quad
    \includegraphics[height=5.5cm]{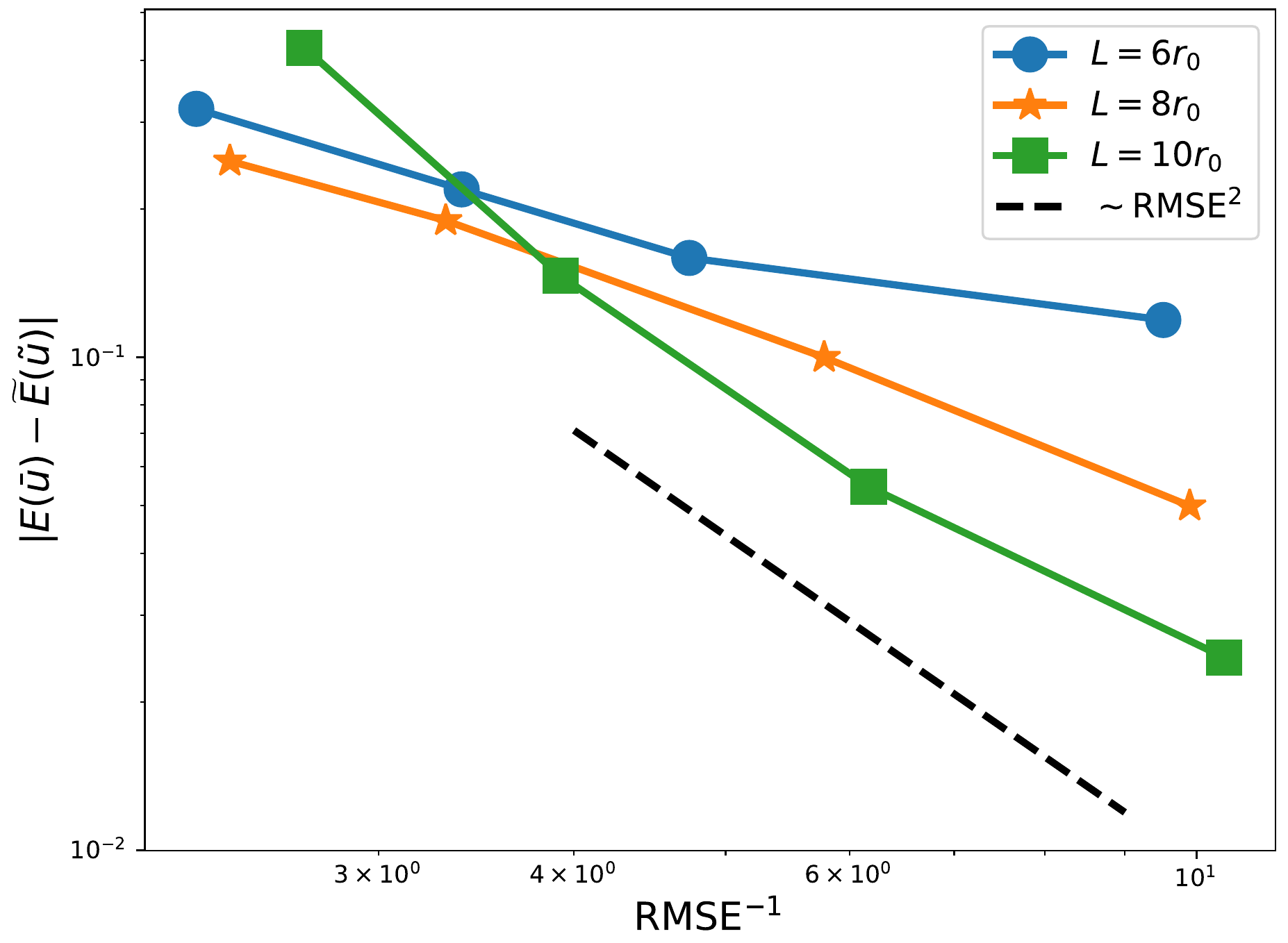}
    \caption{NRL-TB Si model: Geometry error (left) and error in energy (right) v.s. RMSE.}
    \label{fig:errorvsrmse:3D}
\end{figure}

\begin{figure}
	\centering 
		\includegraphics[height=5.5cm]{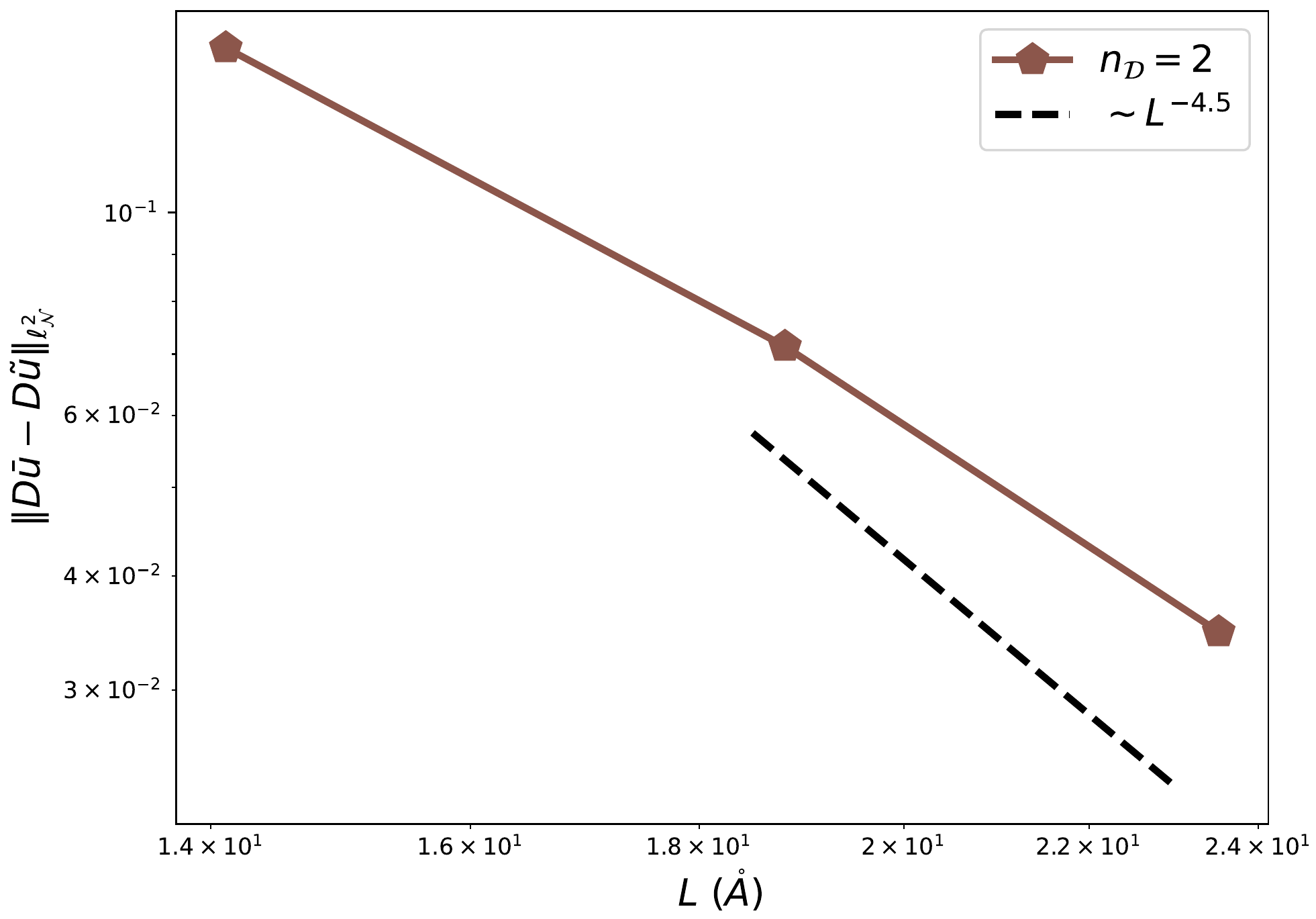}
		\includegraphics[height=5.5cm]{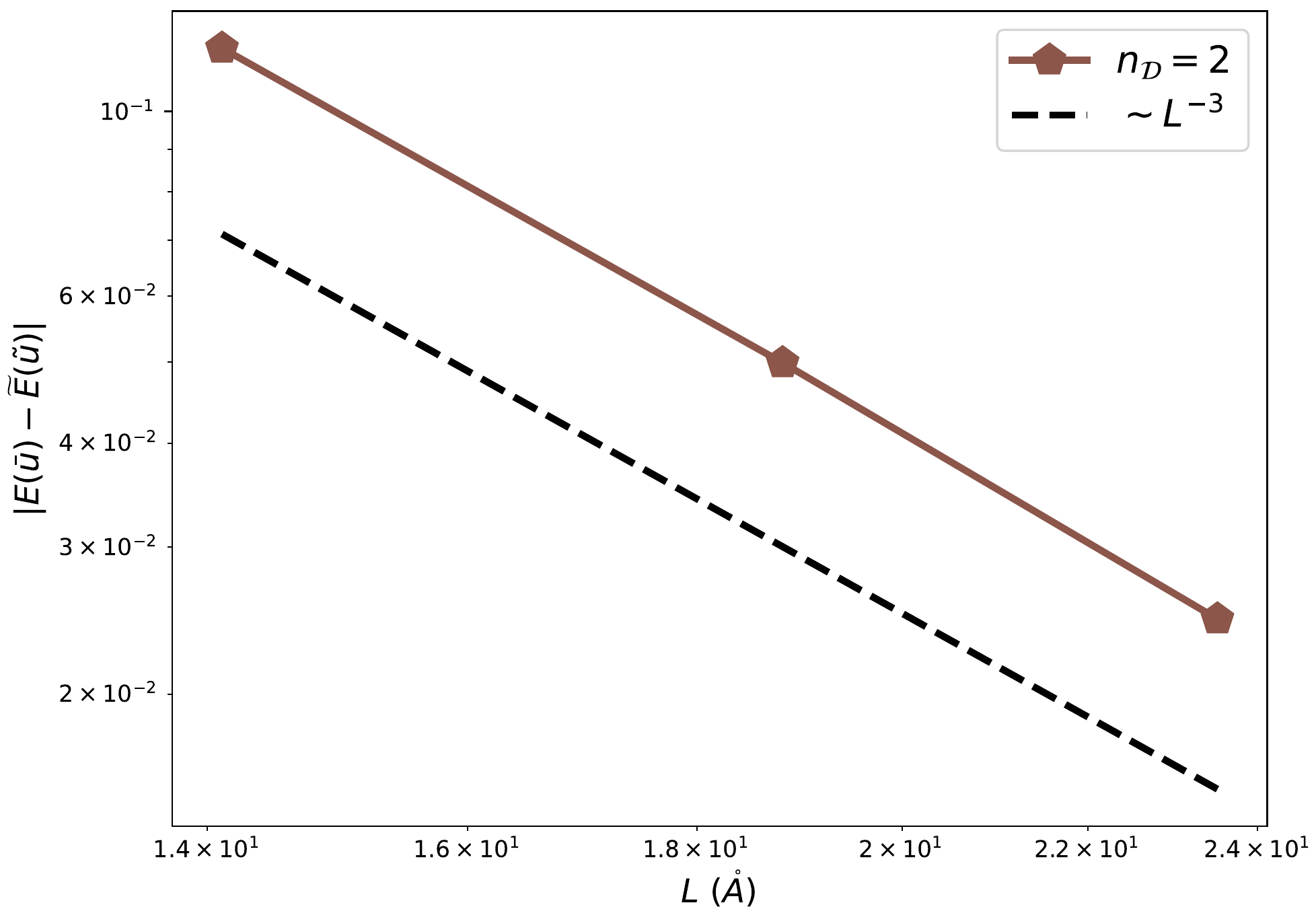}
	\caption{NRL-TB Si model: Convergence rates for geometry error (left) and error in energy (right), ${\rm RMSE}\approx 0.10$.}
	\label{fig:3D_vac}
\end{figure} 

\section{Conclusion}
\label{sec:conclu}
We proposed a framework for a generalisation analysis in a multi-scale setting, and used it to investigate the error propagation from fitting MLIPs on a small training domain to making predictions on a large simulation domain. 
As a prototypical example, we apply the framework to the case of simulating multiple (weakly interacting) point defects in a crystalline solid. Our analysis identifies what observations one should acquire from the reference model to obtain accurate predictions in this case. 
Our theoretical results partially justify existing best practices in the MLIP literature, but also provide more fine-grained qualitative information about how prediction accuracy depends on the choice of training data. 
This approach also suggests a new perspective on how to approach the collection of training data and the design of loss functions. 

Our presentation here is restricted to simple-species Bravais lattices and point defects. Generalisations do require additional technical difficulties to be overcome, but there appears to be no fundamental limitation to extend the method and the results to multi-lattices and a range of other defects in some form. To conclude, we briefly discuss some of these possibilities as well as limitations which deserve further mathematical analysis and algorithmic developments. 

\begin{itemize}
	\item {\it More complex crystalline structures:} As mentioned above, the extension to multi-lattices is conceptually straightforward and the necessary technical details should be addressed in depth.

    \item {\it Straight dislocations:} The extension to straight dislocations appears straightforward applying the techniques of~\cite{2021-qmmm3, Ehrlacher16, 2014-dislift}. We expect that the error estimates for straight dislocations depend not only on the force error but also the {\it matching condition} on the linear elasticity, even the nonlinear elasticity (virials) due to the long-range elastic field. 
    
    \item {\it Grain boundaries or curved dislocation lines:} These more complex crystalline defects require a much more significant degree of extrapolation than point defects or straight dislocations. 
	  Both simulations and rigorous analysis appear to  be both conceptually and technically much more challenging. While our overarching strategies should still apply, it is far less clear whether our methodologies in this paper can be applied directly.
    
    
    \item {\it Uncertainty estimation:} It is common knowledge that MLIPs have a fundamental limitation in that they lack a physical model for the phenomenon being predicted and thus have unknown accuracy when extrapolating beyond their training set. The uncertainty quantification (UQ) capabilities would be included to address this problem. Analyzing the propagation of uncertainty in the training procedure to predicted properties could be understood from a Bayesian statistics perspective, where some recent works~\cite{bartok2022improved, musil2019fast} should provide appropriate references. 
\end{itemize}


\section{Proofs}
\label{sec:proof}

\subsection{Preliminaries}
\label{sec:sub:pre}
In this section, we introduce the concepts needed in the proofs of the main results. We review the framework of~\cite{chen19, Ehrlacher16} to formulate the equilibration of a single point defect ($n_{\D}=1$, vacancy or interstitial) as a variational problem in a discrete energy space and then give the strong stability assumption~\asS, which assumes the existence of a single stable core in the infinite lattice $\L$. 

The deformed configuration of the infinite lattice $\L$ is a map $y:\L\rightarrow\R^d$.
We can decompose the configuration $y$ into
\begin{eqnarray}\label{config_y_u}
y(\ell) = x(\ell) + u(\ell) = \ell + u(\ell) \qquad\forall~\ell\in\Lambda,
\end{eqnarray}
where $x(\ell)=\ell$ is a linear map representing a crystalline reference configuration.
The set of possible atomic configurations is
\begin{align}\label{eq:admiss}
\Adm_{0}(\L) &:= \bigcup_{\mathfrak{m}>0} \Adm_{\mathfrak{m}}(\L)
\qquad \text{with} \\
\Adm_{\mathfrak{m}}(\L) &:= \left\{ y:\L \rightarrow \R^{d}, ~
|y(\ell)-y(m)| > \mathfrak{m} |\ell-m|
\quad\forall~  \ell, m \in \L \right\}, \nonumber
\end{align}
where the parameter $\mathfrak{m}>0$ qualifies the accumulation of atoms.

For site $\ell\in\L$, we define the nearest neighbours set $\mathcal{N}(\ell)$ as 
\begin{align}\label{def1:Nl} 
\mathcal{N}(\ell) :=& \left\{ \, m \in \L \setminus \ell~\Big|~\exists \, a \in \mathbb{R}^{d} \text{ s.t. }
|a - \ell| = |a - m| \leq |a - k| \quad \forall \, k \in \L \, \right\},
\end{align}
which is applied in the definition of the energy norm $\|\cdot\|_{\ell^2_{\mathcal{N}}}$ (cf. \eqref{eq: nn norm}).
We introduce the discrete energy space for infinite lattice
\begin{align}\label{space:UsH}
\UsH(\L) := \big\{u:\L\rightarrow\mathbb{R}^{d} ~\big\lvert~ \|Du\|_{\ell^2_{\mathcal{N}}(\L)}<\infty \big\},
\end{align}
with the associated semi-norm $\|Du\|_{\ell^2_{\mathcal{N}}}$.
We also define the following subspace of compact displacements
\begin{align}\label{space:Uc}
\Us^{\rm c}(\L) := \big\{u:\L\rightarrow\mathbb{R}^{d} ~\big\lvert~ \exists~R >0~{\rm s.t.}~u = {\rm const}~{\rm in}~\L\setminus B_{R}\big\}.
\end{align}
The associated class of admissible displacements is given by
\begin{eqnarray*}
	\Admu(\L):= \big\{ u\in\UsH(\L) ~:~ x +u\in\Adm_0(\L) \big\}.
\end{eqnarray*}

In this paper, we make the following assumptions on the regularity and locality of the site potentials, which has been justified for some basic quantum mechanic models \cite{chen18,chen16,chen19tb, co2020}.
We refer to \cite[\S 2.3 and \S 4]{chen19} for discussions of more general site potentials.
\begin{itemize}
	\label{as:SE:pr}
	\item[\assERL]
	{\it Regularity and locality:}
	For all $\ell \in \L$, $V_{\ell}\big(Du(\ell)\big)$ possesses partial derivatives up to $\mathfrak{n}$-th order with $\n\geq 3$. For $j=1,\ldots,\n$, there exist constants $C_j$ and $\eta_j$ such that
	\begin{eqnarray}
	\label{eq:Vloc}
	\big|V_{\ell,{\bm \rho}}\big(Du(\ell)\big)\big|  \leq
	C_j \exp\Big(-\eta_j\sum^j_{l=1}|{\bm \rho}_l|\Big)
	\end{eqnarray} 
	for all $\ell \in \L$ and ${\bm \rho} \in (\L - \ell)^{j}$. 
\end{itemize}

Although we defined the site potentials on infinite stencils $(\R^d)^{\L-\ell}$, the setting also applies to finite systems or to finite range interactions. 
It is only necessary to assume in this case that the potential $V_{\ell}(\pmb{g})$ does not depend on the reference sites $\pmb{g}_{\rho}$ outside the interaction range.
In particular, we will denote by $V_{\ell}^{\Omega}$ the site potential of a finite system with the reference configuration lying in $\L\cap\Omega$.

Following the results in \cite[Theorem 2.1]{chen19}, the energy-difference functional for infinite lattice reads
\begin{eqnarray}\label{energy-difference}
\E^{\L}(u) := \sum_{\ell\in\Lambda} V_{\ell}\big(Du(\ell)\big).
\end{eqnarray}
The corresponding variational problem for the equilibrium state is
\begin{equation}\label{eq:variational-problem}
\bar{u}^{\rm core} \in \arg\min \big\{ \E^{\L}(u),~u \in \Admu(\L) \big\},
\end{equation}
where ``$\arg\min$'' is understood as the set of local minimizers.

An auxiliary energy-functional needed in the following analysis is the energy of the homogeneous (defect-free) lattice
\begin{eqnarray}\label{energy-homogeneous}
\E^{\rm h}(u) := \sum_{\ell\in\Lhom} V^{\rm h}\big(Du(\ell)\big).
\end{eqnarray}

We are ready to give the assumption on the existence of a single strongly stable core in infinite lattice. 
Let $\bar{u}^{\rm core}$ be the corresponding local minimizer of \eqref{eq:variational-problem}. Then we give the rigorous formulation of \asS~as follows
\begin{eqnarray}\label{eq:def_S}
\exists~\bar{c}_{\L}>0~\textrm{s.t.}~ 
		\big\< \delta^2\E^{\L}(\bar{u}^{\rm core}) v , v\big\> \geq \bar{c}_{\L} \|Dv\|^2_{\ell^2_{\mathcal{N}}} \qquad\forall~v\in\UsH(\L).
\end{eqnarray}

To map the displacements defined on $\L$ to $\L_N^{\rm per}$, we introduce the operator $T_N^{\rm per}: \Us^{1,2}(\L) \rightarrow \Us^{\rm per}_{N}(\L_N)$. One possible construction is given in \cite[Section 7.3]{Ehrlacher16}.
Then we briefly establish two auxiliary stability results based on \asS~that will be needed in our
subsequent analysis.
We include them here for the sake of completeness and their proofs can be found in \cite[Section B.2 and Theorem 7.7]{Ehrlacher16}.

\begin{proposition}[Phonon stability]
\label{pro:phonon}
Suppose that \asS~holds, then there exists a constant $\bar{c}_{\rm hom}$ satisfying $\bar{c}_{\hom} \geq \bar{c}_{\L} > 0$ such that 
\[
\big\< \delta^2 \E^{\rm h}({\bf 0})Dv, Dv\big\> \geq \bar{c}_{\rm hom} \|Dv\|^2_{\ell^2_{\mathcal{N}}} \qquad \forall~v\in\Us^{\rm c}(\Lhom).
\]
\end{proposition}

\begin{proposition}
\label{pro:periodic}
Suppose that \asS~holds and the energy-difference functional $\E$ is defined by \eqref{energy-difference-per}. Then, for $N$ sufficiently large, there exists a constant $\bar{c}>0$ such that 
\[
\big\< \delta^2\E(T_N^{\rm per} \bar{u}^{\rm core}) v , v\big\> \geq \bar{c} \|Dv\|^2_{\ell^2_{\mathcal{N}}(\L_N)} \qquad \forall~v\in\Us^{\per}_N.
\]
\end{proposition}

To conclude this section, we give the decay estimates of the equilibrium state for single point defect \cite[Theorem 3.2]{chen19}: If $\bar{u}^{\rm core}\in\Admu(\L)$ is a strongly stable solution to \eqref{eq:variational-problem} satisfying \asS, then there exists $C > 0$ such that
\begin{eqnarray}\label{eq:ubar-decay}
\big|D\bar{u}^{\rm core}(\ell)\big|_{\mathcal{N}} \leq C \big(1+|\ell|\big)^{-d}.
\end{eqnarray}


\subsection{Proof of the existence results}
\label{sec:sub:thm1}

In this section, we give the detailed proof of Theorem~\ref{them:existence}. 
Recall the definitions introduced in Section~\ref{sec:equilibration} that $\D$ is a set of the positions of the point defects cores in $\L_N$ with the minimum separation distance $L_{\D}$, we assume $L_{\D} \ll N$. We define an approximated solution (predictor) $z:\L_N \rightarrow \R^d$ to the variational problem~\eqref{eq:variational-problem-per} as
\begin{eqnarray}\label{eq:z}
z(\ell) := \sum_{\ell_i \in \D} \Pi_R \bar{u}^{\rm core}(\ell - \ell_i) \qquad \forall \ell \in \L_N, 
\end{eqnarray}
where the truncation operator $\Pi_R$ is defined by \eqref{eq:trun_op_def} with radius $R=L_{\D}/3$.

The following Lemma provides an estimate on the residual of such an approximated solution in terms of $L_{\D}$.

\begin{lemma}\label{le:thm1:cons}
Suppose $z$ is the approximated solution to the variational problem~\eqref{eq:variational-problem-per} as defined in \eqref{eq:z} with truncation radius $R=L_{\D}/3$. Then, there exists a constant $L_0>0$ such that, for $L_{\D} > L_0$, 
\begin{eqnarray}
\big|\<\delta\E(z), v\>\big| \leq C \sqrt{n_{\D}} \cdot L_{\D}^{-d/2}\cdot \|D v\|_{\ell^2_{\mathcal{N}}(\L_N)},
\end{eqnarray}
where the constant $C$ is independent of $N, n_{\D}, L_{\D}$.
\end{lemma}

\begin{proof}

Let $r:=R+1=L_{\D}/3+1$. For any $v \in \Us^{\per}_N$, we define
\begin{eqnarray}\label{eq:decomp_v_res}
v_i := \Pi^{\ell_i}_{r} v \quad \textrm{for}~i=1,\ldots,n_{\D}, \quad \textrm{and} \quad v_0:=v-\sum_{i=1}^{n_{\D}}v_i,
\end{eqnarray}
where the defect core truncation operator $\Pi^{\ell_i}_{r}$ is defined by \eqref{eq:trun_op_def}. For each $i=1,\ldots,n_{\D}$, $\Pi^{\ell_i}_r$ is extended periodically with respect to $\L_N$ since $v \in \Us_N^{\rm per}$. Lemma \ref{le:tr} implies that $\|D v_i\|_{\ell^2_{\mathcal{N}}(\L_{N})} \leq C\|D v\|_{\ell^2_{\mathcal{N}}(\L_{N})}$ for $i=1,\ldots,n_{\D}$.

We then decompose the residual into three parts
\begin{align}\label{eq:T12}
    \<\delta\E(z), v\> &= \sum_{i=0}^{n_{\D}}\<\delta\E(z), v_i\> \nonumber \\
    &= \<\delta \E(z), v_0\> + \sum^{n_{\D}}_{i=1} \<\delta\E(z)-\delta\E\big(T_N^{\rm per} \bar{u}^{\rm core}(\cdot - \ell_i)\big), v_i\> \nonumber \\
    &\hskip2.5cm + \sum^{n_{\D}}_{i=1} \<\delta\E\big(T_N^{\rm per} \bar{u}^{\rm core}(\cdot - \ell_i)\big), v_i\> \nonumber \\[1ex]
    &=: T_1 + T_2 + T_3,
\end{align}
where the operator $T_N^{\rm per}$ maps the displacements from $\Us^{1,2}(\L)$ to $\Us^{\rm per}_{N}(\L_N)$. For the term $T_1$, we obtain that
\begin{align}\label{eq:T1}
    \big|T_1 \big| &= \big|\<\delta\E(z), v_0\>\big| = \Big| \sum_{\ell\in\L_N} \sum_{\rho\in\Lhom\setminus0}V_{\ell, \rho}({\bf 0})\cdot D_{\rho}v_0(\ell) \Big| \nonumber \\[1ex]
    &\leq \Big|\sum_{\ell \in {{\rm supp}}(v_0)} \F^{\rm h}_{\ell}({\bf 0})\cdot v_0(\ell) \Big| + \Big| \sum_{\ell\in \L\cap(\bigcup_{i=1}^{n_{\D}} B_{3r/4}(\ell_i))} \sum_{\substack{\rho\in \Lhom\setminus0,\\ \ell+\rho \in {\rm supp}(v_0)}} V_{\ell, \rho}({\bf 0})\cdot v_0(\ell+\rho)\Big| \nonumber \\[1ex]
    &\leq Ce^{-\kappa r} \|Dv\|_{\ell^2_{\mathcal{N}}(\L_N)},
\end{align}
where $r=L_{\D}/3+1$ and the last inequality follows from the locality of site potentials~\assERL~with a constant $\kappa>0$.

To estimate $T_2$, by using Lemma \ref{le:tr}, it is straightforward to obtain that 
\begin{align}\label{eq:T2}
\big|T_2 \big| &\leq \sum_{i=1}^{n_{\D}} \big|\<\delta\E(z)-\delta\E\big(T_N^{\rm per}\bar{u}^{\rm core}(\cdot - \ell_i)\big), v_i\>\big| \nonumber \\[1ex]
&\leq CM_1\|D\Pi_R \bar{u}^{\rm core} - D \bar{u}^{\rm core}\|_{\ell^2_{\mathcal{N}}(\L_r)}\cdot \sum_{i=1}^{n_{\D}} \|D v_i\|_{\ell^2_{\mathcal{N}}(\L_r)} \nonumber \\[1ex]
& \leq CM_1 \sqrt{n_{\D}} \cdot R^{-d/2} \cdot \|D v\|_{\ell^2_{\mathcal{N}}(\L_N)},
\end{align}
where $M_1$ is the uniform Lipschitz constant of $\delta\E$ since $\E$ is $(\n-1)$-times continuously differentiable with respect to the $\|D\cdot\|_{\ell^2_{\mathcal{N}}}$ norm \cite{chen19, 2021-qmmm3}.

The term $T_3$ can be similarly estimated by using \cite[Lemma 7.6]{Ehrlacher16}
\begin{eqnarray}\label{eq:T3}
\big|T_3\big| \leq C\sqrt{n_{\D}} \cdot N^{-d/2} \cdot \|D v\|_{\ell^2_{\mathcal{N}}(\L_N)}.
\end{eqnarray}

Combing \eqref{eq:T12}, \eqref{eq:T1}, \eqref{eq:T2} and \eqref{eq:T3}, by exploiting the assumption that $L_{\D} \ll N$, for $L_{\D}$ sufficiently large, we have
\begin{eqnarray}
\big|\<\delta\E(z), v\>\big| \leq C \sqrt{n_{\D}} \cdot L_{\D}^{-d/2}\cdot \|D v\|_{\ell^2_{\mathcal{N}}(\L_N)},
\end{eqnarray}
which yields the stated result.
\end{proof}

We then proceed to prove that $\delta^2\E(z)$ is positive, where $z$ is given by \eqref{eq:z}. This result employs the ideas similar to those used in the proofs of \cite[Theorem 7.7]{Ehrlacher16} and \cite[Lemma 5.2]{2014-dislift}, modified here to an periodic setting and extended to cover the case of multiple point defects. 

\begin{lemma}\label{le:thm1:stab}
Suppose $z$ is the approximated solution to the variational problem~\eqref{eq:variational-problem-per} as defined in \eqref{eq:z}. Then, there exists a constant $L_0>0$ such that, for $L_{\D} > L_0$, there exists $\bar{c}_{L_{\D}} \geq \bar{c}/2$ so that
\begin{eqnarray}\label{eq:exist:stab}
\<\delta^2\E(z)v, v\> \geq \bar{c}_{L_{\D}} \|Dv\|^2_{\ell^2_{\mathcal{N}}(\L_{N})} \qquad \forall~v\in \Us^{\per}_N.  
\end{eqnarray}
\end{lemma}
\begin{proof}
We argue by contradiction. Suppose that there exists no $L_0$ satisfying \eqref{eq:exist:stab}, it follows that there exists a sequence of multiple point defects configurations $\D_k:=\{\ell_i^k\}^{n_{\D}}_{i=1}$ such that: (1) $L_k:=L_{\D_k} \rightarrow \infty$ as $k \rightarrow \infty$; (2) for all $k$, let $z_k$ be denoted as the approximated solution defined by \eqref{eq:z}, we have
\begin{equation*}
    \bar{c}_k := \inf_{\|Dv\|^2_{\ell^2_{\mathcal{N}}}=1}\<\delta^2\E(z_k)v, v\> < \bar{c}/2 .
\end{equation*}
Hence, let $v_k \in \Us_N^{\per}$ be a sequence of test functions such that $\|D v_k\|_{\ell^2_{\mathcal{N}}}=1$, we can obtain
\begin{eqnarray}\label{eq:Sbound}
\bar{c}_k \leq \<\delta^2\E(z_k) v_k, v_k\> \leq \bar{c}_k + k^{-1}.
\end{eqnarray}
We now employ the result in \cite[Lemma 7.9]{Ehrlacher16}. This states that there exists a sequence of radii, $r_k \rightarrow \infty$, for which we may also assume $r_k \leq L_{k}/3 $, so that for each $i=1,\ldots,n_{\D}$, 
\[
w^{i}_k := \Pi^{\ell^k_i}_{r_k} v_k, \quad \textrm{and defining}~~ w^0_k := v_k - \sum_{i=1}^{n_{\D}} w^i_k,
\]
where the construction follows from \eqref{eq:decomp_v_res}. It follows that
\begin{align}\label{eq:splitS}
\<\delta^2\E(z_k) v_k, v_k\> &= \sum_{i,j=0}^{n_{\D}} \<\delta^2\E(z_k) w_k^i, w_k^j\> \nonumber \\
&= \<\delta^2\E(z_k) w_k^0, w_k^0\> + \sum_{i=1}^{n_{\D}}  \<\delta^2\E(z_k) w_k^i, w_k^i\> + 2\sum_{i=1}^{n_{\D}} \<\delta^2\E(z_k) w_k^0, w_k^i\> \nonumber \\[1ex]
&=: S_1 + S_2 + S_3,
\end{align}
where we have ensured that supp$\{w^i_k\}$ for $i=1,\ldots,n_{\D}$ only overlaps with supp$\{w^0_k\}$ by choosing $r_k \leq L_k/3$, and hence all other cross-terms vanish.

For term $S_1$, we have 
\begin{align}
    \<\delta^2\E(z_k) w_k^0, w_k^0\> &= \big\<\big(\delta^2\E(z_k)-\delta^2\E({\bf 0})\big) w_k^0, w_k^0\big\> + \<\delta^2\E({\bf 0}) w_k^0, w_k^0\> \nonumber \\[1ex]
    &\geq \big(\bar{c}_{\rm hom} - M_2 \sqrt{n_{\D}} \cdot r^{-d/2}_k \big) \|D w_k^0\|^2_{\ell^2_{\mathcal{N}}(\L_N)},
\end{align}
where the last inequality follows from the  Proposition \ref{pro:phonon}.

The term $S_2$ can be estimated similarly from Lemma \ref{le:pi:stab}
\begin{align}
    \big|S_2\big| \geq \big(\bar{c}_{R} - M_2 \sqrt{n_{\D}} \cdot r^{-d/2}_k \big) \|D v_k\|^2_{\ell^2_{\mathcal{N}}(\L_N)}.
\end{align}

For the cross-terms $S_3$, by assuming $r_k \leq L_k/3$, we deduce that 
\[
\<\delta^2\E(z_k) w_k^0, w_k^i\> = \<\delta^2\E(z_k) (v_k-w_k^i), w_k^i\>.
\]
Using the techniques in the proof of \cite[Lemma 5.2]{2014-dislift}, we can obtain that, for each $i$,
\begin{eqnarray}\label{eq:S3}
\<\delta^2\E(z_k) w_k^0, w_k^i\> \rightarrow 0, \qquad \textrm{as}~ k\rightarrow\infty.
\end{eqnarray}

Combining the estimates from \eqref{eq:splitS} to \eqref{eq:S3}, we have
\[
\<\delta^2\E(z_k) v_k, v_k\> \geq (\bar{c}_R - \epsilon_k) \|D v_k\|^2_{\ell^2_{\mathcal{N}}(\L_N)} + \epsilon_k,
\]
where $\epsilon_k \rightarrow 0$ as $k \rightarrow \infty$. Furthermore, Lemma \ref{le:pi:stab} implies that for sufficiently large $R$ we have $\bar{c}_{R}\geq \bar{c}/2 > 0$. Hence, together with \eqref{eq:Sbound}, we can obtain 
\[
\bar{c}_k + k^{-1} \geq \<\delta^2\E(z_k) v_k, v_k\> \geq \bar{c}_R \geq \bar{c}/2,
\]
so for $k$ sufficiently large, it contradicts the assumption that $\bar{c}_k<\bar{c}/2$ for all $k$.
\end{proof}

We are ready to prove the Theorem~\ref{them:existence}.

\begin{proof}[Proof of Theorem \ref{them:existence}]
Applying Lemma \ref{le:ift} with the results in Lemma \ref{le:thm1:cons} and Lemma \ref{le:thm1:stab}, we can state that there exist $L_0>0$, where $\D$ satisfies $L_{\D}\geq L_0$, and $z$ is an approximated solution defined by \eqref{eq:z} corresponding to $\D$. It follows that for any $v \in \Us_N^{\per}$, there exists $\omega \in \Us_N^{\per}$ with $\|D\omega\|_{\ell^2_{\mathcal{N}}(\L_N)} \leq C\sqrt{n_{\D}} \cdot L_{\D}^{-d/2}$ such that
\begin{align*}
\<\delta\E(z+\omega), v\>=0, \quad \<\delta^2\E(z+\omega)v, v\> \geq \frac{\bar{c}}{4}\|Dv\|^2_{\ell^2_{\mathcal{N}}}.
\end{align*}
Writing $\bar{u}:=z+\omega$ yields the stated result.
\end{proof}

\subsection{Proof of the generalisation analysis}
\label{sec:sub:thm2}


In this section, we give the detailed proof of the generalisation analysis (Theorem \ref{them:geometry}), which is the main result in this paper. 

\begin{proof}
Applying the framework of the {\it a priori} error estimates in \cite{chen17, co2011, colz2012, colz2016}, we mainly divide the proof into several steps in order to apply the inverse function theorem (cf.~Lemma \ref{le:ift}).

{\it 1. Stability:} For any $v \in \Us^{\per}_N$, we consider the stability of
\begin{align}\label{eq:thm2:splitS}
\<\delta^2 \widetilde{\E}(\bar{u})v, v\> &=  \<\delta^2 \E(\bar{u})v, v \> + \big( \<\delta^2 \widetilde{\E}(\bar{u})v, v \> - \<\delta^2 \E(\bar{u})v, v \> \big) \nonumber \\[1ex]
&=: S_1 + S_2.
\end{align}

From the results in Theorem \ref{them:existence}, we can obtain that $\bar{u}$ is strongly stable, namely there exists $\bar{c}>0$ such that
\[
S_1:=\<\delta^2\E(\bar{u})v, v\> \geq \frac{\bar{c}}{4}\|Dv\|^2_{\ell^2_{\mathcal{N}}}.
\]

Similarly as shown in \eqref{eq:decomp_v_res}, we split the test function $v$. For $L \leq L_{\D}$, let $r:=L/3+1$, we define
\begin{eqnarray}\label{eq:decomp_v}
v_i := \Pi^{\ell_i}_{r} v \quad \textrm{for}~i=1,\ldots,n_{\D}, \quad \textrm{and} \quad v_0:=v-\sum_{i=1}^{n_{\D}}v_i,
\end{eqnarray}
Lemma \ref{le:tr} implies that $\|D v_i\|_{\ell^2_{\mathcal{N}}(\L_{N})} \leq C\|D v\|_{\ell^2_{\mathcal{N}}(\L_{N})}$ for $i=1,\ldots,n_{\D}$. 

The term $S_2$ can be further split into three parts
\begin{align}\label{eq:thm2:splitS2}
\big\<\big(\delta^2\widetilde{\E}(\bar{u})-\delta^2\E(\bar{u})\big) v, v\big\> =&~ \sum_{i,j=0}^{n_{\D}} \big\<\big(\delta^2\widetilde{\E}(\bar{u})-\delta^2\E(\bar{u})\big) v_i, v_j \big\> \nonumber \\[1ex]
=& ~\big\<\big(\delta^2\widetilde{\E}(\bar{u})-\delta^2\E(\bar{u})\big) v_0, v_0\big\> \nonumber \\
&+ \sum_{i=1}^{n_{\D}}  \big\<\big(\delta^2\widetilde{\E}(\bar{u})-\delta^2\E(\bar{u})\big) v_i, v_i \big\> \nonumber \\
&+ 2\sum_{i=1}^{n_{\D}} \big\<\big(\delta^2\widetilde{\E}(\bar{u})-\delta^2\E(\bar{u})\big) v_0, v_i \big\> \nonumber \\[1ex]
=:& ~S_{21} + S_{22} + S_{23},
\end{align}
where we have ensured that supp$\{v_i\}$ for $i=1,\ldots,n_{\D}$ only overlaps with supp$\{v_0\}$ by the choice of $r$ and therefore all other cross-terms vanish.

For the term $S_{21}$, 
for $L$ sufficiently large and $t\in[0,1]$, we can Taylor expand the $\delta^2\E$ at the reference configuration
\begin{align}
\big\<\big(\delta^2\widetilde{\E}(\omega)-\delta^2\E(\omega)\big)v_0, v_0\big\> &= \big\<\big(\delta^2 \widetilde{\E}({\bf 0})-\delta^2 \E({\bf 0})\big)v_0, v_0\big\> + \big\<\big(\delta^3 \widetilde{\E}(t\omega)-\delta^3 \E(t\omega)\big)\omega v_0, v_0\big\> \nonumber \\[1ex]
&=: S^{\rm (a)}_{21} + S^{\rm (b)}_{21}. 
\end{align}
We first estimate $S^{\rm (a)}_{21}$ by
\begin{align}\label{eq:s21a}
|S^{\rm (a)}_{21}| \leq  C \varepsilon^{\rm FC}_{\rm hom} \cdot \|D v\|^2_{\ell^2_{\mathcal{N}}(\L_N)}.
\end{align}
%
For the term $S^{\rm (b)}_{21}$, similarly we can obtain
\begin{eqnarray}
|S^{\rm (b)}_{21}| \leq C L^{-d/2} \cdot \|D v\|^2_{\ell^2_{\mathcal{N}}(\L_N)}.
\end{eqnarray}

To estimate $S_{22}$, recall the definition of the predictor \eqref{eq:z} and the construction of $v_i$, for each $i=1,\ldots,n_{\D}$ and $L$ sufficiently large, we have
\begin{align}
&\big\<\big(\delta^2\widetilde{\E}(\bar{u})-\delta^2\E(\bar{u})\big) v_i, v_i\big\> \nonumber \\[1ex]
=&~ \big\<\big(\delta^2\widetilde{\E}(\Pi_{R}\bar{u}^{\rm core}(\cdot-\ell_i) )-\delta^2\E(\Pi_{R}\bar{u}^{\rm core}(\cdot-\ell_i))\big) v_i, v_i\big\> \nonumber \\[1ex]
=&~ \big\<\big(\delta^2\widetilde{\E}(\Pi_{R}\bar{u}^{\rm core}(\cdot-\ell_i) )-\delta^2\widetilde{\E}(\bar{u}_L(\cdot-\ell_i))\big) v_i, v_i\big\> \nonumber \\
&+ \big\<\big(\delta^2\widetilde{\E}(\bar{u}_L(\cdot-\ell_i) )-\delta^2\E(\bar{u}_L(\cdot-\ell_i))\big) v_i, v_i\big\> \nonumber \\
&+ \big\<\big(\delta^2\E(\bar{u}_L(\cdot-\ell_i) )-\delta^2\E(\Pi_{R}\bar{u}^{\rm core}(\cdot-\ell_i))\big) v_i, v_i\big\> \nonumber \\[1ex]
\leq&~ C (\varepsilon^{\rm FC} + \|D \Pi_{R}\bar{u}^{\rm core} - D\bar{u}_L\|_{\ell^2_{\mathcal{N}}(\L_L)}) \cdot \|D v\|^2_{\ell^2_{\mathcal{N}}(\L_N)} \nonumber \\[1ex]
\leq&~ C (\varepsilon^{\rm FC}+L^{-d/2}) \cdot \|D v\|^2_{\ell^2_{\mathcal{N}}(\L_N)},
\end{align}
where the last inequality follows from \cite[Theorem 2.1]{2018-uniform}.

Noting that $\<\delta^2\E(u) v_0, v_i\> = \<\delta^2\E(u) (v-v_i), v_i\>$, the term $S_{23}$ can be estimated similarly by
\begin{align}\label{eq:S23}
    \big\<\big(\delta^2\widetilde{\E}(\bar{u})-\delta^2\E(\bar{u})\big) v_0, v_i\big\>
    \leq C (\varepsilon^{\rm FC}+L^{-d/2}) \cdot \|D v\|^2_{\ell^2_{\mathcal{N}}(\L_N)}.
\end{align}

Hence, combining the estimates from \eqref{eq:thm2:splitS} to \eqref{eq:S23}, for $L$ sufficiently large and the matching conditions $\varepsilon^{\rm FC}, \varepsilon^{\rm FC}_{\rm hom}$ sufficiently small, we have
\begin{eqnarray}\label{eq:thm2:stab}
\<\delta^2\widetilde{\E}(\bar{u})v, v\> \geq \frac{\bar{c}}{8}\|Dv\|^2_{\ell^2_{\mathcal{N}}(\L_N)}.
\end{eqnarray}

{\it 2. Consistency:}
We estimate the consistency error, for any $v \in \Us_{N}^{\rm per}$, by
\begin{align}\label{eq:splitcons}
    \<\delta \widetilde{\E}(\bar{u}), v\> &= \<\delta \widetilde{\E}(\bar{u})-\delta\E(\bar{u}), v\> \nonumber \\
    &= \sum^{n_{\D}}_{i=1} \<\delta \widetilde{\E}(\bar{u})-\delta\E(\bar{u}), v_i\> + \<\delta \widetilde{\E}(\bar{u})-\delta\E(\bar{u}), v_0\> \nonumber \\
    &=: T_1 + T_2,
\end{align}
where $v_i, i=0,\ldots,n_{\D}$, are constructed by \eqref{eq:decomp_v}.

To estimate $T_1$, for each $i=1, \ldots, n_{\D}$, we denote
\begin{align}\label{eq:cons1}
    T^{\rm (i)}_{1} := \<\delta \widetilde{\E}(\bar{u})-\delta\E(\bar{u}), v_i\> 
    = \big\<\delta\widetilde{\E}\big(\Pi_{R}\bar{u}^{\rm core}(\cdot-\ell_i)\big)-\delta\E\big(\Pi_{R}\bar{u}^{\rm core}(\cdot-\ell_i)\big), v_i\big\>. 
\end{align}
Applying~\cite[Theorem 2.1]{2018-uniform}, we have $\|D \Pi_{R}\bar{u}^{\rm core} - D\bar{u}_L\|_{\ell^2_{\mathcal{N}}(\L_L)} \leq C L^{-d/2}$. Hence, given $\delta>0$, for $L$ sufficiently large, $\Pi_{R}\bar{u}^{\rm core} \in B_{\delta}(\bar{u}_L)$. 
%
Recalling the definition of $\varepsilon^{\rm F}$, we can obtain 
\begin{eqnarray}\label{eq:T1a}
|T_1| \leq \sum^{n_{\D}}_{i=1} |T^{\rm (i)}_1| \leq  C \sqrt{n_{\D}} \cdot \varepsilon^{\rm F} \cdot \|D v\|_{\ell^2_{\mathcal{N}}(\L_N)}.
\end{eqnarray}

To estimate $T_2$, we note that $\bar{u}$ is smooth in supp$\{v_0\}$ for $L$ sufficiently large according to Theorem \ref{them:existence}. Hence, we can Taylor expand $\delta \widetilde{\E}$ and $\delta \E$ at the reference configuration
\begin{align}\label{eq:cons2}
    \big\< \delta \widetilde{\E}(\omega) -  \delta \E(\omega), v_0 \big\>
    &= \big\<\big(\delta^2 \widetilde{\E}({\bf 0}) - \delta^2 \E({\bf 0}) \big)\omega, v_0 \big\> + \big\<\big(\delta^3 \widetilde{\E}(t\omega) - \delta^3 \E(t\omega) \big)(\omega)^2, v_0 \big\> \nonumber \\[1ex]
    &=: T_{21} + T_{22},
\end{align}
where $t\in[0,1]$. The term $T_{21}$ can be bounded by
\begin{eqnarray}
|T_{21}| \leq C \sqrt{n_{\D}} \cdot L^{-d/2} \varepsilon^{\rm FC}_{\rm hom} \cdot \|D v\|_{\ell^2_{\mathcal{N}}(\L_N)}.
\end{eqnarray}
For $T_{22}$, it can be estimated similarly
\begin{eqnarray}\label{eq:T22}
|T_{22}| \leq C \sqrt{n_{\D}} \cdot L^{-3d/2} \cdot \|D v\|_{\ell^2_{\mathcal{N}}(\L_N)}.
\end{eqnarray}

Hence, combining from \eqref{eq:splitcons} to \eqref{eq:T22}, we can obtain 
\begin{eqnarray}\label{eq:thm2:cons}
\<\delta \widetilde{\E}(\bar{u}), v\> \leq C \sqrt{n_{\D}} \cdot \big( \varepsilon^{\rm F} + L^{-d/2} \varepsilon^{\rm FC}_{\rm hom} + L^{-3d/2}\big) \cdot \|D v\|_{\ell^2_{\mathcal{N}}(\L_N)}.
\end{eqnarray}

{\it 3. Application of inverse function theorem:} With the stability \eqref{eq:thm2:stab} and consistency \eqref{eq:thm2:cons}, we can apply the inverse function theorem (Lemma \ref{le:ift}) to obtain, for $L$ sufficiently large and the matching conditions defined in \eqref{eq:ED} and \eqref{eq:FCDhom} sufficiently small, the existence of a solution $\tilde{u}$ to \eqref{eq:variational-problem-approx}, and the estimate
\[
\|D\bar{u} - D\tilde{u}\|_{\ell^2_{\mathcal{N}}(\L_N)} \leq C^{\rm G} \sqrt{n_{\D}} \cdot \big( \varepsilon^{\rm F} + L^{-d/2} \varepsilon^{\rm FC}_{\rm hom} + L^{-3d/2} \big),
\]
where $C^{\rm G}$ is independent of $N, n_{\D}, L$. This completes the proof of \eqref{eq:geoerr}.

{\it 4: Error in energy:} Next, we estimate the error in the energy. Recall the definition of the predictor $z$ by \eqref{eq:z}, we first spilt the error in energy into two parts
\begin{eqnarray}\label{eq:splitE}
\big| \E(\bar{u}) - \widetilde{\E}(\tilde{u}) \big| \leq \big| \E(\bar{u}) - \E(z)\big| + \big| \E(z) - \widetilde{\E}(\tilde{u}) \big| =: E_1 + E_2
\end{eqnarray}
The term $E_1$ can be bounded by
\begin{align}\label{eq:E1}
\big|\E(\bar{u}) - \E(z)\big| &= \Big| \int_0^1 \big\<\delta\E\big((1-s)\bar{u}+sz\big), \bar{u}-z \big\> \ds \Big| \nonumber \\[1ex]
&= \Big| \int_0^1 \big\<\delta\E\big((1-s)\bar{u}+sz\big)-\delta\E(\bar{u}), \bar{u}-z \big\> \ds \Big| \nonumber \\[1ex]
&\leq C M_1 \cdot \|D \bar{u} - Dz \|^2_{\ell^2_{\mathcal{N}}(\L_N)} \leq C n_{\D} \cdot L^{-d},
\end{align}
where $M_1$ is the uniform Lipschitz constant of $\delta\E$.

To estimate $E_2$, by applying the technique used in \eqref{eq:E1}, similarly we can obtain  
\begin{align}\label{eq:E2}
\big| \widetilde{\E}(\tilde{u}) - \E(z)\big| &\leq \big| \widetilde{\E}(\tilde{u}) - \widetilde{\E}(z)\big| + \big| \widetilde{\E}(z) - \E(z)\big| \nonumber \\[1ex]
&\leq C \widetilde{M}_1 \cdot \|D\tilde{u} - Dz\|^2_{\ell^2_{\mathcal{N}}(\L_N)} + \big| \widetilde{\E}(z) - \E(z)\big| \nonumber \\[1ex]
&\leq C n_{\D} \cdot \big( \|D\tilde{u} - D\bar{u}\|^2_{\ell^2_{\mathcal{N}}(\L_N)} + \|D\bar{u} - Dz\|^2_{\ell^2_{\mathcal{N}}(\L_N)} + \varepsilon^{\rm E} \big) \nonumber \\[1ex]
&\leq C n_{D} \cdot \big(\|D\tilde{u} - D\bar{u}\|^2_{\ell^2_{\mathcal{N}}(\L_N)} + L^{-d} + \varepsilon^{\rm E} \big),
\end{align}
where $\widetilde{M}_1$ is the uniform Lipschitz constant of $\delta\widetilde{\E}$.

Combining \eqref{eq:splitE}, \eqref{eq:E1} and \eqref{eq:E2}, we obtain
\[
\big| \E(\bar{u}) - \widetilde{\E}(\tilde{u}) \big| \leq C^{\rm E} n_{\D} \cdot \Big( \big( \varepsilon^{\rm F} + L^{-d/2} \varepsilon^{\rm FC}_{\rm hom} + L^{-3d/2} \big)^2 + L^{-d} + \varepsilon^{\rm E} \Big),
\]
which completes the proof of Theorem \ref{them:geometry}.
\end{proof}

\appendix

\section{Auxiliary results}
\label{sec:anx}
We first review a quantitative version of the inverse function theorem, adapted from \cite[Lemma B.1]{mlco2013}.

\begin{lemma}\label{le:ift}
Let $X, Y$ be Hilbert spaces, $w \in X$, $F\in C^2(B^{X}_{R}(w);Y)$ with Lipschitz continuous Hessian, $\|\delta^2F(x)-\delta^2F(y)\|_{L(X,Y)} \leq M\|x-y\|_{X}$ for any $x, y \in B^X_{R}(w)$. Furthermore, suppose that there exist constants $\mu, r>0$ such that
\[
\<\delta^2 F(w)v, v\> \geq \mu\|v\|^2_X, \quad \|\delta F(w)\|_Y \leq r, \quad \textrm{and}~\frac{2Mr}{\mu^2}<1,
\]
then there exists a locally unique $\bar{w} \in B^{X}_R(w)$ such that $\delta F(\bar{w})=0, \|w - \bar{w}\|_{X} \leq \frac{2r}{\mu}$ and 
\[
\<\delta^2 F(\bar{w})v, v\> \geq \big(1 - \frac{2Mr}{\mu^2}\big)\mu \|v\|^2_X.
\]
\end{lemma}

Following \cite{Ehrlacher16} we define a family of truncation operators $\{\Pi_{R}^{\ell_i}\}_{\ell_i \in \D}$, which we will apply to the single strongly stable core $\bar{u}^{\rm core}$. Let $\eta \in C^1(\R^d;[0,1])$ be a cut-off function satisfying $\eta(x)=1$ for $|x|\leq 4/6$ and $\eta(x)=0$ for $|x|\geq 5/6$.

We denote $\T_{\L}$ as the {\it canonical triangulation} of $\R^d$ whose nodes are the reference sites $\L$ (cf.~\cite[Section 2.1]{Ehrlacher16}). Let $Iu$ be the piecewise affine interpolant of $u$ with respect to $\T_{\L}$, and $A_{R}:=B_{5R/6}\setminus B_{4R/6}$ be an annulus, then we can define the truncation operator by
\begin{eqnarray}\label{eq:trun_op_def}
\Pi^{\ell_i}_{R}u(\ell):= \eta\Big(\frac{\ell - \ell_i}{R}\Big)\big(u(\ell)-a^{\ell_i}_{R} \big), \quad \textrm{where} \quad a^{\ell_i}_{R}:=\bbint_{\ell_i + A_R} Iu(x)\dx.
\end{eqnarray}
In addition, we define $\Pi_{R}:=\Pi^{0}_{R}$.

In this paper, we consider multiple point defects in a finite domain $\L_N$ with periodic boundary condition. Hence, for $R \ll N$, we extend $\Pi_R$ periodically with respect to $\L_N$. 
%
We state the following result concerning the approximation property of the truncation operator $\Pi_R$, which follows from results in \cite{Ehrlacher16}.

\begin{lemma}\label{le:tr}
Let $v\in\UsH(\L)$, there exists $C>0$ such that, for $R$ sufficiently large,
\begin{align}\label{eq:tr}
    \|D \Pi_{R} v - Dv\|_{\ell^2_{\mathcal{N}}} &\leq C\|Dv\|_{\ell^2_{\mathcal{N}}(\L \setminus B_{R/2})}.
\end{align}
In particular, if $\bar{u}^{\rm core}$ is the core corrector satisfying~\asS, 
we have
\begin{align}\label{eq:trR}
    \|D \Pi_R \bar{u}^{\rm core} - D\bar{u}^{\rm core}\|_{\ell^2_{\mathcal{N}}} \leq CR^{-d/2},
\end{align}
where $C$ is independent of $R$ and $\D$.
\end{lemma}
\begin{proof}
The estimate \eqref{eq:tr} is simply a restatement of \cite[Lemma 4.5]{2014-dislift}, while the second estimate \eqref{eq:trR} then follows immediately from \eqref{eq:ubar-decay}.
\end{proof}

Next, we show that the assumption \asS~implies that $\delta^2\E(\Pi_R \bar{u}^{\rm core})$ is also positive for sufficiently large $R$, which follows directly from Proposition \ref{pro:periodic}. 

\begin{lemma}\label{le:pi:stab}
Suppose $\bar{u}^{\rm core}$ satisfies \asS, then for sufficiently large $R$, there exist a constant $\bar{c}_R>0$ such that
\begin{eqnarray}
\<\delta^2\E(\Pi_R \bar{u}^{\rm core})v, v\> \geq \bar{c}_R \|D v\|^2_{\ell^2_{\mathcal{N}}(\L_N)},
\end{eqnarray}
and $\bar{c}_{R} \rightarrow \bar{c}$ given in Proposition \ref{pro:periodic}, as $R \rightarrow \infty$.
\end{lemma}
\begin{proof}
For any $v \in \Us^{\per}_{N}$, 
\begin{align}
    \<\delta^2\E(\Pi_R \bar{u}^{\rm core})v, v\> &= \big\< \big(\delta^2\E(\Pi_R \bar{u}^{\rm core}) - \delta^2\E(T^{\rm per}_{N} \bar{u}^{\rm core}) \big)v, v \big\> + \<\delta^2\E(T_N^{\rm per} \bar{u}^{\rm core})v, v \big\> \nonumber \\[1ex]
    &\geq \big( \bar{c} - M_2\|D\Pi_R \bar{u}^{\rm core} - D\bar{u}^{\rm core}\|_{\ell^2_{\mathcal{N}}(\L_N)} \big) \|D v\|^2_{\ell^2_{\mathcal{N}}(\L_N)} \nonumber \\[1ex]
    &\geq (\bar{c} - M_2R^{-d/2}) \|D v\|^2_{\ell^2_{\mathcal{N}}(\L_N)},
\end{align}
where the mapping $T_N^{\rm per}: \Us^{1,2}(\L) \rightarrow \Us^{\rm per}_{N}(\L_N)$ is introduced in Section~\ref{sec:sub:pre} and $M_2$ is the uniform Lipschitz constant of $\delta^2\E$ since $\E$ is $(\n-1)$-times continuously differentiable. As $R\rightarrow \infty$, we can obtain the stated result.
\end{proof}

\section{Numerical supplements}
\label{sec:numer_supp}

We give the illustration of the simulation domain for $n_{\D}=3,4$ considered in this paper in Figure~\ref{fig:multivac_app}.

\begin{figure}[!htb]
    \centering
		\includegraphics[scale=0.5]{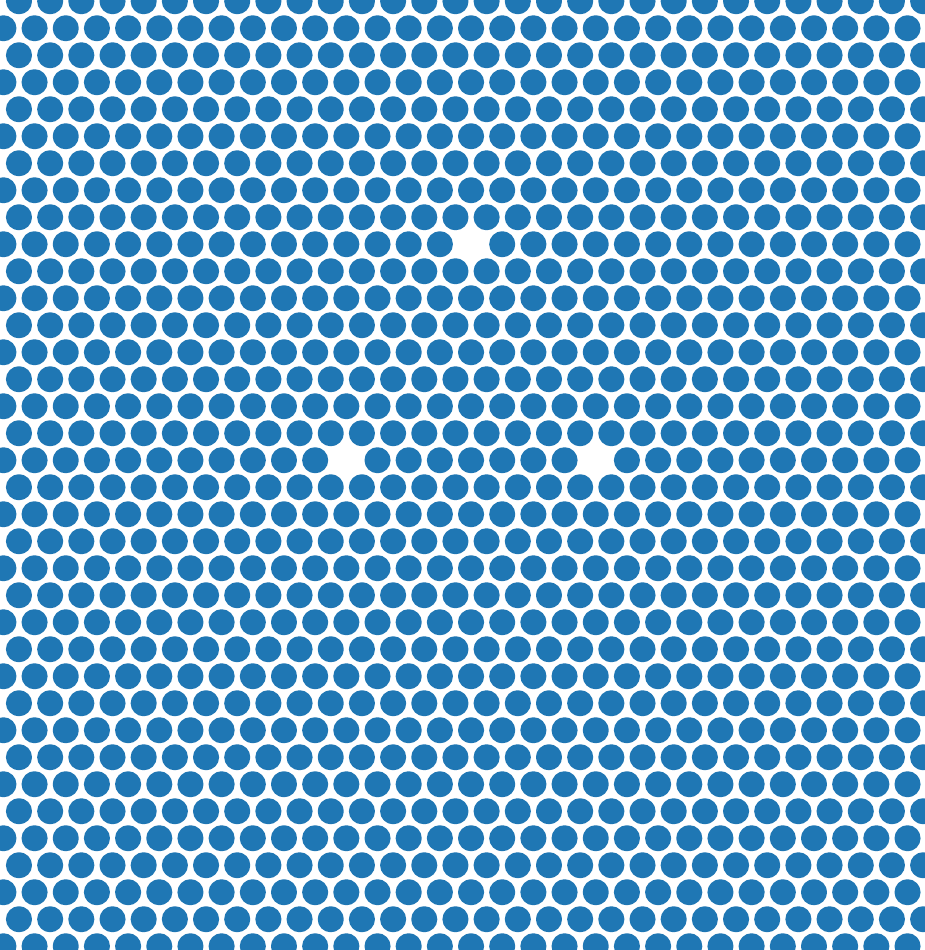}
		\qquad \quad
		 \includegraphics[scale=0.5]{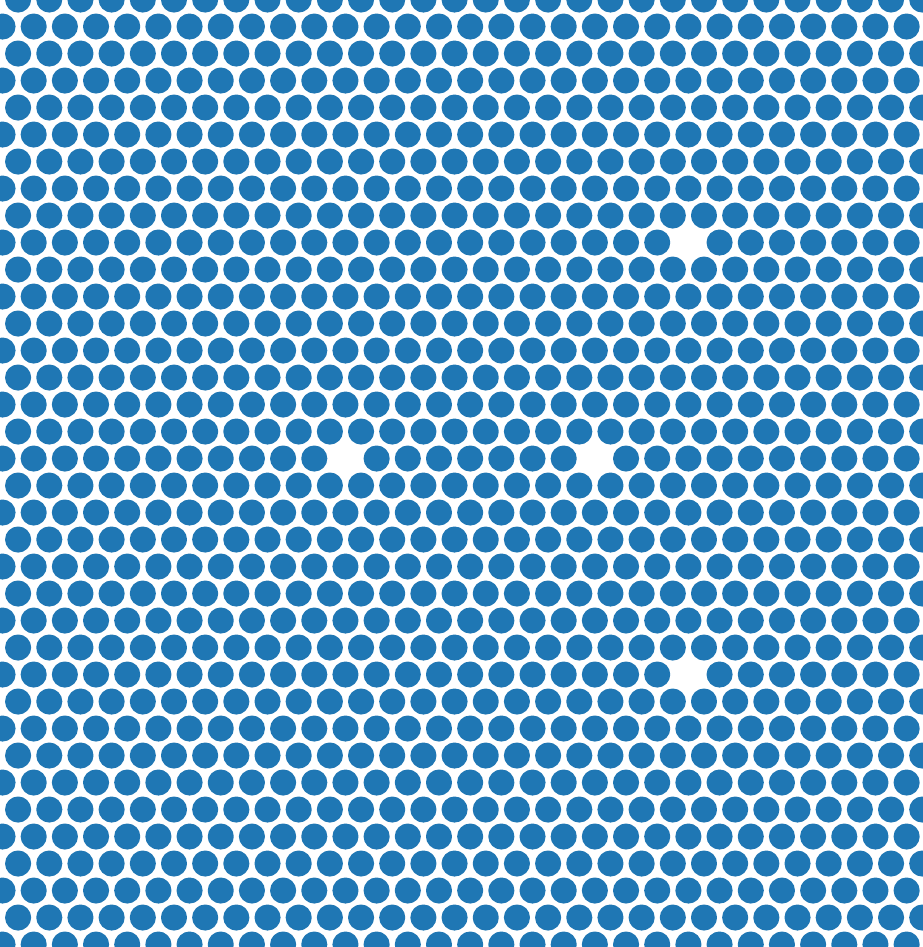}
    \caption{2D toy model: Illustration of the simulation domains for multi-vacancies cases with $n_{\D}=3$ (left) and $n_{\D}=4$ (right).}
    \label{fig:multivac_app}
\end{figure}

Table~\ref{table-params} presents the detailed fitting parameters, fitting accuracy in terms of RMSE, force constant error and fitting time for the case of two separated vacancies, where the total number of the observations and parameters $\{c_B\}_{B\in \pmb{B}}$ are denoted by $\# {\pmb O}$ and $\# \pmb{B}$, respectively.

\begin{table}
\footnotesize
\setlength\tabcolsep{14pt}
\renewcommand\arraystretch{0.9}
	\begin{center}
    {\bf Fitting accuracy on measurements and fitting time T}\\
	\vskip0.3cm
	\begin{tabular}{|cccccc|}
	\hline
		{} & {$\# {\pmb B}$} & {$\efit$ (meV)} & {$\ffit$ (eV/$\mathring{\mathrm{A}}$)} & {$\varepsilon^{\rm FC}_{\rm hom}$ (\%)} & {T (s)}\\[1mm]
		\hline
		\multirow{5}{*}{\shortstack{$L=4r_0$\\$\# {\pmb O}=6800$}}&51 &0.329 &0.110 &10.7 &26.5 \\
		&110 &0.288 &0.076 &8.9 &27.8 \\ &211 &0.206 &0.045 &7.7 &26.8 \\ &383 &0.231 &0.039 &6.8 &30.7 \\ &1119 &0.150 &0.036 &6.3 &38.7 
		\\
		\hline
		\multirow{5}{*}{\shortstack{$L=5r_0$\\$\# {\pmb O}=14000$}}&110 &0.364 &0.107 &11.0 &56.6 \\
		&211 &0.261 &0.072 &8.7 &54.5 \\ &383 &0.208 &0.048 &7.4 &55.8 \\ &668 &0.183 &0.043 &7.5 &58.4 \\ &1119 &0.143 &0.038 &5.1 &77.9 
		\\
		\hline
		\multirow{5}{*}{\shortstack{$L=6r_0$\\$\# {\pmb O}=17600$}}&110 &0.401 &0.106 &11.9 &67.3 \\
		&383 &0.274 &0.074 &8.3 &68.4 \\ &668 &0.165 &0.038 &6.3 &70.6 \\ &1119 &0.139 &0.035 &5.4 &80.7 \\ &1445 &0.127 &0.032 &4.9 &96.0 
		\\
		\hline
		\multirow{5}{*}{\shortstack{$L=7r_0$\\$\# {\pmb O}=28400$}}&211 &0.298 &0.113 &11.4 &114.8 \\
		&383 &0.297 &0.102 &9.6 &116.2 \\ &1119 &0.159 &0.050 &7.6 &119.6 \\ &1445 &0.095 &0.040 &6.8 &135.9 \\ &1809 &0.085 &0.034 &5.9 &172.8 
		\\
		\hline
		\multirow{5}{*}{\shortstack{$L=8r_0$\\$\# {\pmb O}=41600$}}&383 &0.271 &0.128 &13.8 &166.9 \\
		&668 &0.209 &0.114 &9.0 &162.2 \\ &1119 &0.127 &0.084 &7.4 &166.0 \\ &1809 &0.092 &0.038 &5.9 &199.8 \\ &2849 &0.102 &0.034 &5.8 &256.6 
		\\
		\hline
	\end{tabular}
	\end{center}
	\caption{2D toy model: The fitting accuracy (RMSE), force constant error and fitting time for the case of two separated vacancies, where $r_0$ is the {\it rescaled} lattice constant.
	}
	\label{table-params}
\end{table}

Figure~\ref{fig:intvac_app} shows the illustration of the simulation domain for interstitial-vacancies case with $n_{\D}=4$ considered in the numerical experiments.

\begin{figure}[!htb]
    \centering
\includegraphics[scale=0.55]{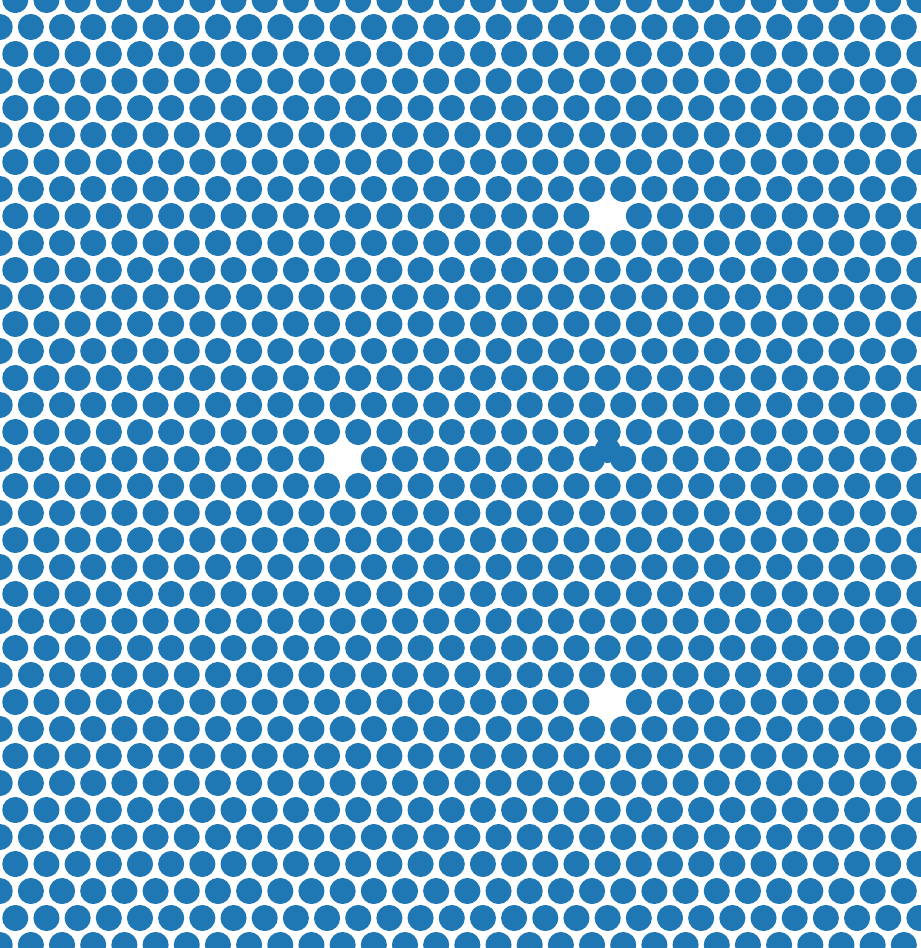}
    \caption{2D toy model: Illustration of the simulation domain for interstitial-vacancies case with $n_{\D}=4$.}
    \label{fig:intvac_app}
\end{figure}

\section{The Atomic Cluster Expansion}
\label{sec:ACE}

Following \cite{2019-ship1}, we briefly introduce the construction of the ACE potential. 
Given $\mathcal{N} \in \N$, we first write the ACE site potential in the form of an {\it atomic body-order expansion}, $\displaystyle V^{\rm ACE} \big(\{\pmb{g_j}\}\big) = \sum_{N=0}^{\mathcal{N}} \frac{1}{N!}\sum_{j_1 \neq \cdots \neq j_N} V_N(\pmb{g}_{j_1}, \cdots, \pmb{g}_{j_N})$, where the $N$-body potential $V_N : \R^{dN} \rightarrow \R$ can be approximated by using a tensor product basis \cite[Proposition 1]{2019-ship1},
\begin{align*}
\phi_{\pmb{k\ell m}}\big(\{\pmb{g}_j\}_{j=1}^N\big) :=\prod_{j=1}^N\phi_{k_j\ell_j m_j}(\pmb{g}_j) 
\quad & {\rm with}\quad 
\phi_{k\ell m}(\pmb{r}):=P_k(r)Y^{m}_{\ell}(\hat{r}) , ~~\pmb{r}\in\R^d,~r=|\pmb{r}|,~\hat{r}=\pmb{r}/r ,
\end{align*}
where $P_k,~k=0,1,2,\cdots$ are radial basis functions (for example, Jacobic polynomials), and $Y_{\ell}^m,~\ell = 0,1,2,\cdots,~m=-\ell,\cdots,\ell$ are the complex spherical harmonics.
The basis functions are further symmetrised to a permutation invariant form,
\begin{eqnarray*}
\tilde{\phi}_N = \sum_{(\pmb{k,\ell,m})~{\rm ordered}} \sum_{\sigma\in S_N} \phi_{\pmb{k\ell m}} \circ \sigma ,
\end{eqnarray*}  
where $S_N$ is the collection of all permutations, and by $\sum_{(\pmb{k,\ell,m})~{\rm ordered}}$ we mean that the sum is over all lexicographically ordered tuples $\big( (k_j, \ell_j, m_j) \big)_{j=1}^N$.
The next step is to incorporate the invariance under point reflections and rotations
\begin{eqnarray*}
\mathcal{B}_{\pmb{k\ell} i} = \sum_{\pmb{m}\in\mathcal{M}_{\pmb{\ell}}}\mathcal{U}_{\pmb{m}i}^{\pmb{k\ell}} \sum_{\sigma\in S_N} \phi_{\pmb{k\ell m}} \circ \sigma 
\quad{\rm with}\quad \mathcal{M}_{\pmb{\ell}} = \big\{\pmb{\mu}\in\Z^N ~|~ -\ell_{\alpha}\leq\mu_{\alpha}\leq\ell_{\alpha} \big\} ,
\end{eqnarray*}
where the coefficients $\mathcal{U}_{mi}^{k\ell}$ are given in \cite[Lemma 2 and Eq. (3.12)]{2019-ship1}.
It was shown in \cite{2019-ship1} that the basis defined above is explicit but computational inefficient. The so-called ``density trick" technique used in \cite{Bart10, Drautz19, Shapeev16} can transform this basis into one that is computational efficient. The alternative basis is 
\begin{eqnarray*}
B_{\pmb{k\ell} i} = \sum_{\pmb{m}\in\mathcal{M}_{\pmb{\ell}}}\mathcal{U}_{\pmb{m}i}^{\pmb{k\ell}} A_{\pmb{nlm}}
\quad{\rm with~the~correlations}\quad A_{\pmb{nlm}}:= \prod_{\alpha=1}^{N} \sum_{j=1}^{J} \phi_{n_\alpha l_\alpha m_{\alpha}}(\pmb{g}_j),
\end{eqnarray*}
which avoids both the $N!$ cost for symmetrising the basis as well as the $C_J^N$ cost of summation over all order $N$ clusters within an atomic neighbourhood.
The resulting basis set is then defined by
\begin{eqnarray}
\pmb{B}_N := \big\{ B_{\pmb{k\ell} i}  ~|~ (\pmb{k},\pmb{\ell})\in\N^{2N}~{\rm ordered}, ~\sum_{\alpha} \ell_{\alpha}~{\rm even},~ i=1,\cdots,\pmb{n}_{\pmb{k\ell}}\big\},
\end{eqnarray}
where $\pmb{n}_{\pmb{k\ell}}$ is the rank of body-orders (see \cite[Proposition 7 and Eq. (3.12)]{2019-ship1}).

Once the finite symmetric polynomial basis set $\pmb{B} \subset \bigcup^{\mathcal{N}}_{N=1} \pmb{B}_N$ is constructed, the ACE site potential can be expressed as
\begin{align}\label{ships:energy}
V^{\rm ACE}(\pmb{g}; \{c_B\}_{B\in\pmb{B}}) = \sum_{B\in\pmb{B}} c_B B(\pmb{g}) 
\end{align}
with the coefficients $c_{B}$. The corresponding force of this potential is denoted by $\F^{\rm ACE}$.

The family of potentials are systematically improvable (see \cite[\S 6.2]{2019-ship1}): by increasing the body-order, cutoff radius and polynomial degree they are in principle capable of representing an arbitrary many-body potential energy surface to within arbitrary accuracy.

\section{A semi-empirical QM model: NRL tight binding}
\label{sec:NRL}

The NRL tight binding model is first developed in \cite{cohen94}. The energy levels are determined by the generalised eigenvalue problem
\begin{eqnarray}\label{NRL:diag_Ham}
\mathcal{H}(y)\psi_s = \lambda_s\mathcal{M}(y)\psi_s 
\qquad\text{with}\quad \psi_s^{\rm T}\mathcal{M}(y)\psi_s = 1,
\end{eqnarray}
where $\mathcal{H}$ is the hamiltonian matrix and $\mathcal{M}(y)$ is the overlap matrix. The NRL hamiltonian and overlap matrices are construct both from hopping elements as well as on-site matrix elements as a function of the local environment. For carbon and silicon they are parameterised as follows (for other elements the parameterisation is similar): 

To define the on-site terms, each atom $\ell$ is assigned a pseudo-atomic density 
\begin{eqnarray*}
	\rho_{\ell} := \sum_{k}e^{-\lambda^2 r_{\ell k}}  \fc(r_{\ell k}),
\end{eqnarray*}
where the sum is over all of the atoms $k$ within the cutoff $\Rc$ of atom $\ell$, $\lambda$ is a fitting parameter, $\fc$ is a cutoff function
\begin{eqnarray*}
	\fc(r) = \frac{\theta(\Rc-r)}{1+\exp\big((r-\Rc)/l_c + L_c\big)} ,
\end{eqnarray*}
with $\theta$ the step function, and the parameters $l_c=0.5$, $L_c=5.0$ for most elements.
Although, in principle, the on-site terms should have off-diagonal elements, but this would lead to additional computational challenges that we wished to avoid. The NRL model follows traditional practice and only include the diagonal terms.
Then, the on-site terms for each atomic site $\ell$ are given by
\begin{eqnarray}\label{NRLonsite}
\mathcal{H}(y)_{\ell\ell}^{\upsilon\upsilon}
:= a_{\upsilon} + b_{\upsilon}\rho_{\ell}^{2/3} + c_{\upsilon}\rho_{\ell}^{4/3} + d_{\upsilon}\rho_{\ell}^{2},
\end{eqnarray}
where $\upsilon=s,p$, or $d$ is the index for angular-momentum-dependent atomic orbitals and $(a_\upsilon)$, $(b_\upsilon)$, $(c_\upsilon)$, $(d_\upsilon)$ are fitting parameters.  The on-site elements for the overlap matrix are simply taken to be the identity matrix.

The off-diagonal NRL Hamiltonian entries follow the formalism of  Slater and Koster who showed in \cite{Slater54} that all two-centre (spd) hopping integrals can be constructed from ten independent ``bond integral'' parameters $h_{\upsilon\upsilon'\mu}$, where
\begin{eqnarray*}
	(\upsilon\upsilon'\mu) = ss\sigma,~sp\sigma,~pp\sigma,~pp\pi,
	~sd\sigma,~pd\sigma,~pd\pi,~dd\sigma,~dd\pi,~{\rm and}~dd\delta.
\end{eqnarray*}
The NRL bond integrals are given by
\begin{eqnarray}\label{NRLhopping}
h_{\upsilon\upsilon'\mu}(r) 
:= \big(e_{\upsilon\upsilon'\mu} + f_{\upsilon\upsilon'\mu}r + g_{\upsilon\upsilon'\mu} r^2 \big) e^{-h_{\upsilon\upsilon'\mu}r} \fc(r)
\end{eqnarray}
with fitting parameters $e_{\upsilon\upsilon'\mu}, f_{\upsilon\upsilon'\mu},  g_{\upsilon\upsilon'\mu}, h_{\upsilon\upsilon'\mu}$. The matrix elements $\mathcal{H}(y)_{\ell k}^{\upsilon\upsilon'}$ are constructed from the $h_{\upsilon\upsilon'\mu}(r)$ by a standard procedure~\cite{Slater54}.

The analogous bond integral parameterisation of the overlap matrix 
is given by  
\begin{eqnarray}\label{NRLhopping-M}
m_{\upsilon\upsilon'\mu}(r) 
:= \big(\delta_{\upsilon\upsilon'} + p_{\upsilon\upsilon'\mu} r + q_{\upsilon\upsilon'\mu} r^2 + r_{\upsilon\upsilon'\mu} r^3 \big) e^{-s_{\upsilon\upsilon'\mu} r} \fc(r)
\end{eqnarray}
with the fitting parameters $(p_{\upsilon\upsilon'\mu}), (q_{\upsilon\upsilon'\mu}), (r_{\upsilon\upsilon'\mu}), (s_{\upsilon\upsilon'\mu})$ and $\delta_{\upsilon\upsilon'}$ the Kronecker delta function. 

The fitting parameters in the foregoing expressions are determined by fitting to some high-symmetry first-principle calculations: In the NRL method, a database of eigenvalues (band structures) and total energies were constructed for several crystal structures at  several volumes. Then the parameters are chosen such that the eigenvalues and energies in the database are reproduced.
For practical simulations, the parameters for different elements can be found in~\cite{papaconstantopoulos15}.

\bibliographystyle{plain}
\bibliography{bib.bib}

\begin{thebibliography}{10}

\bibitem{2019-ship1}
M.~Bachmayr, G.~Csanyi, G.~Dusson, R.~Drautz, S.~Etter, C.~van~der Oord, and
  C.~Ortner.
\newblock Atomic cluster expansion: Completeness, efficiency and stability.
\newblock {\em J. Comp. Phys.}, 454:110946, 2022.

\bibitem{bartok2022improved}
A.~Bart{\'o}k and J.~Kermode.
\newblock Improved uncertainty quantification for gaussian process regression
  based interatomic potentials.
\newblock {\em arXiv preprint arXiv:2206.08744}, 2022.

\bibitem{Bart10}
A.~Bart\'{o}k, M.~Payne, R.~Kondor, and G.~Cs\'{a}nyi.
\newblock Gaussian approximation potentials: {T}he accuracy of quantum
  mechanics, without the electrons.
\newblock {\em Phys. Rev. Lett.}, 104:136403, 2010.

\bibitem{behler07}
J.~Behler and M.~Parrinello.
\newblock Generalized neural-network representation of high-dimensional
  potential-energy surfaces.
\newblock {\em Phys. Rev. Lett.}, 98:146401, 2007.

\bibitem{Braams09}
B.~Braams and J.~Bowman.
\newblock Permutationally invariant potential energy surfaces in
  highdimensionality.
\newblock {\em Int. Rev. Phys. Chem.}, 28:577--606, 2009.

\bibitem{2018-uniform}
J.~Braun and C.~Ortner.
\newblock Sharp uniform convergence rate of the supercell approximation of a
  crystalline defect.
\newblock {\em SIAM J. Numer. Anal.}, 58, 2020.

\bibitem{chan1987rank}
F.~Chan.
\newblock Rank revealing {QR} factorizations.
\newblock {\em Linear Algebra Appl.}, 88:67--82, 1987.

\bibitem{chen18}
H.~Chen, J.~Lu, and C.~Ortner.
\newblock Thermodynamic limit of crystal defects with finite temperature tight
  binding.
\newblock {\em Arch. Ration. Mech. Anal.}, 230:701--733, 2018.

\bibitem{chen19}
H.~Chen, F.Q. Nazar, and C.~Ortner.
\newblock Geometry equilibration of crystalline defects in quantum and
  atomistic descriptions.
\newblock {\em Math. Models Methods Appl. Sci.}, 29:419--492, 2019.

\bibitem{chen16}
H.~Chen and C.~Ortner.
\newblock {QM/MM} methods for crystalline defects. {Part 1: L}ocality of the
  tight binding model.
\newblock {\em Multiscale Model. Simul.}, 14:232--264, 2016.

\bibitem{chen17}
H.~Chen and C.~Ortner.
\newblock {QM/MM} methods for crystalline defects. {Part 2: C}onsistent energy
  and force-mixing.
\newblock {\em Multiscale Model. Simul.}, 15:184--214, 2017.

\bibitem{chen19tb}
H.~Chen, C.~Ortner, and J.~Thomas.
\newblock Locality of interatomic forces in tight binding models for
  insulators.
\newblock {\em ESAIM: Math. Model. Numer. Anal.}, 54:2295--2318, 2020.

\bibitem{2021-qmmm3}
H.~Chen, C.~Ortner, and Y.~Wang.
\newblock {QM/MM} methods for crystalline defects. part 3: Machine-learned
  interatomic potentials.
\newblock {\em ArXiv e-prints}, 2106.14559, 2021.

\bibitem{cohen94}
R.~Cohen, M.~Mehl, and D.~Papaconstantopoulos.
\newblock Tight-binding total-energy method for transition and noble metals.
\newblock {\em Phys. Rev. B}, 50:14694--14697, 1994.

\bibitem{Daw1984a}
M.~S. Daw and M.~I. Baskes.
\newblock Embedded-atom method: {D}erivation and application to impurities,
  surfaces, and other defects in metals.
\newblock {\em Phys. Rev. B}, 29:6443--6453, 1984.

\bibitem{Drautz19}
R.~Drautz.
\newblock Atomic cluster expansion for accurate and transferable interatomic
  potentials.
\newblock {\em Phys. Rev. B}, 99:014104, 2019.

\bibitem{Ehrlacher16}
V.~Ehrlacher, C.~Ortner, and A.~Shapeev.
\newblock Analysis of boundary conditions for crystal defect atomistic
  simulations.
\newblock {\em Arch. Ration. Mech. Anal.}, 222:1217--1268, 2016.

\bibitem{gitACE}
C.~Ortner {\it et al}.
\newblock {ACE.jl.git}.
\newblock https://github.com/ACEsuit/ACE.jl.

\bibitem{gitSKTB}
C.~Ortner {\it et al}.
\newblock {SKTB.jl.git}.
\newblock https://github.com/cortner/SKTB.jl.git.

\bibitem{finnis03}
M.~Finnis.
\newblock {\em Interatomic Forces in Condensed Matter}.
\newblock Oxford University Press, Oxford, 2003.

\bibitem{galli1992large}
G.~Galli and M.~Parrinello.
\newblock Large scale electronic structure calculations.
\newblock {\em Phys. Rev. Lett.}, 69(24):3547, 1992.

\bibitem{2014-dislift}
T.~Hudson and C.~Ortner.
\newblock Analysis of stable screw dislocation configurations in an anti-plane
  lattice model.
\newblock {\em SIAM J. Math. Anal.}, 41:291--320, 2015.

\bibitem{kohanoff2006electronic}
J.~Kohanoff.
\newblock {\em Electronic structure calculations for solids and molecules:
  theory and computational methods}.
\newblock Cambridge university press, 2006.

\bibitem{kotliar2006electronic}
G.~Kotliar, S.~Savrasov, K.~Haule, V.~Oudovenko, O.~Parcollet, and
  C.~Marianetti.
\newblock Electronic structure calculations with dynamical mean-field theory.
\newblock {\em Rev. Mod. Phys.}, 78(3):865, 2006.

\bibitem{LJ}
J.E. Lennard-Jones.
\newblock On the determination of molecular fields.
\newblock {\em Proc. R. Soc. Lond. A}, 106:463--477, 1924.

\bibitem{mlco2013}
M.~Luskin and C.~Ortner.
\newblock Atomistic-to-continuum-coupling.
\newblock {\em Acta Numerica}, 22:397--508, 2013.

\bibitem{lysogorskiy2021performant}
Y.~Lysogorskiy, C.~Oord, A.~Bochkarev, S.~Menon, M.~Rinaldi, T.~Hammerschmidt,
  M.~Mrovec, A.~Thompson, G.~Cs{\'a}nyi, C.~Ortner, et~al.
\newblock Performant implementation of the atomic cluster expansion (pace) and
  application to copper and silicon.
\newblock {\em Npj Comput. Mater.}, 7(1):1--12, 2021.

\bibitem{mehl96}
M.~Mehl and D.~Papaconstantopoulos.
\newblock Applications of a tight-binding total-energy method for transition
  and noble metals: {E}lastic constants, vacancies, and surfaces of monatomic
  metals.
\newblock {\em Phys. Rev. B}, 54:4519--4530, 1996.

\bibitem{musil2019fast}
F.~Musil, M.~Willatt, M.~Langovoy, and M.~Ceriotti.
\newblock Fast and accurate uncertainty estimation in chemical machine
  learning.
\newblock {\em J. Chem. Theory Comput.}, 15(2):906--915, 2019.

\bibitem{co2011}
C.~Ortner.
\newblock A priori and a posteriori analysis of the quasinonlocal
  quasicontinuum method in 1d.
\newblock {\em Math. Comp.}, 80:1265--1285, 2011.

\bibitem{co2020}
C.~Ortner and J.~Thomas.
\newblock Point defects in tight binding models for insulators.
\newblock {\em Math. Models Methods Appl. Sci.}, 30:2753--2797, 2020.

\bibitem{colz2012}
C.~Ortner and L.~Zhang.
\newblock Construction and sharp consistency estimates for atomistic/continuum
  coupling methods with general interfaces: {A} 2d model problem.
\newblock {\em SIAM J. Numer. Anal.}, 50:2940--2965, 2012.

\bibitem{colz2016}
C.~Ortner and L.~Zhang.
\newblock Atomistic/continuum blending with ghost force correction.
\newblock {\em SIAM J. Sci. Comput.}, 38:A346--A375, 2016.

\bibitem{papaconstantopoulos15}
D.A. Papaconstantopoulos.
\newblock {\em Handbook of the Band Structure of Elemental Solids, From Z = 1
  To Z = 112}.
\newblock Springer New York, 2015.

\bibitem{saad2010numerical}
Y.~Saad, J.~Chelikowsky, and S.~Shontz.
\newblock Numerical methods for electronic structure calculations of materials.
\newblock {\em SIAM Rev.}, 52(1):3--54, 2010.

\bibitem{Shapeev16}
A.~Shapeev.
\newblock Moment tensor potentials: {A} class of systematically improvable
  interatomic potentials.
\newblock {\em Multiscale Model. Simul.}, 14:1153--1173, 2016.

\bibitem{Slater54}
J.C. Slater and G.F. Koster.
\newblock Simplified {LCAO} method for the periodic potential problem.
\newblock {\em Phys. Rev.}, 94:1498--1524, 1954.

\bibitem{Stillinger1985Computer}
F.H. Stillinger and T.A. Weber.
\newblock Computer simulation of local order in condensed phases of silicon.
\newblock {\em Phys. Rev. B}, 31:5262--5271, 1985.

\bibitem{oord19}
C.~van~der Oord, G.~Cs\'{a}nyi, G.~Dusson, and C.~Ortner.
\newblock Regularised atomic body-ordered permutation-invariant polynomials for
  the construction of interatomic potentials.
\newblock {\em Mach. Learn.: Sci. Technol.}, 1:015004, 2020.

\end{thebibliography}
\end{document}